\documentclass[11pt,a4paper]{article}
\usepackage{pdfsync} 

\usepackage{amsfonts,amsmath,amsthm}
\newtheorem{theorem}{Theorem}
\numberwithin{theorem}{section}
\newtheorem{lemma}[theorem]{Lemma}
\newtheorem{proposition}[theorem]{Proposition}

\theoremstyle{definition}

\numberwithin{equation}{section}
\usepackage{amssymb}
\usepackage{hyperref}
\usepackage{graphicx}
\usepackage{wrapfig}
\usepackage{subfig}
\usepackage{epsfig}
\usepackage{float,mathrsfs}
\usepackage{enumerate}
\usepackage{enumitem}
\usepackage{bbm}

\usepackage{algpseudocode}
\usepackage[font=scriptsize]{caption}
\counterwithout{figure}{section}
\usepackage{color}
\definecolor{refkey}{rgb}{0.9451,0.2706,0.4941}\definecolor{labelkey}{rgb}{0.9451,0.2706,0.4941}

\definecolor{darkred}{RGB}{139,0,0}
\definecolor{darkgreen}{RGB}{0,100,0}
\definecolor{darkmagenta}{RGB}{139,0,139}
\definecolor{gray}{RGB}{180,180,180}

\newcommand{\setu}{{\mathrm{\mathfrak{u}}}}

\newcommand{\bsx}{{\boldsymbol{x}}}
\newcommand{\bsy}{{\boldsymbol{y}}}

\newcommand{\bst}{{\boldsymbol{t}}}

\newcommand{\bsgamma}{{\boldsymbol{\gamma}}}

\newcommand{\bbN}{{\mathbb{N}}}
\newcommand{\bbR}{{\mathbb{R}}}

\newcommand{\bsalpha}{{\boldsymbol{\alpha}}}

\newcommand{\bsnu}{{\boldsymbol{\nu}}}
\newcommand{\bsmu}{{\boldsymbol{\mu}}}

\newcommand{\bsb}{{\boldsymbol{b}}}

\newcommand{\mask}[1]{{}}

\usepackage{pgfplots}
\pgfplotsset{compat=1.18}

\title{Quasi-Monte Carlo for partial differential equations\\with generalized Gaussian input uncertainty}

\author{Philipp A.~Guth\footnotemark[2]\and Vesa Kaarnioja\footnotemark[3]}

\allowdisplaybreaks

\begin{document}\maketitle

\begin{abstract}
There has been a surge of interest in uncertainty quantification for parametric partial differential equations (PDEs) with Gevrey regular inputs. The Gevrey class contains functions that are infinitely smooth with a growth condition on the higher-order partial derivatives, but which are nonetheless not analytic in general. Recent studies by Chernov and L\^{e} (\emph{Comput.~Math.~Appl.}, 2024, and \emph{SIAM J.~Numer.~Anal.}, 2024) as well as Harbrecht, Schmidlin, and Schwab (\emph{Math.~Models Methods Appl.~Sci.}, 2024) analyze the setting wherein the input random field is assumed to be uniformly bounded with respect to the uncertain parameters. In this paper, we relax this assumption and allow for parameter-dependent bounds. The parametric inputs are modeled as generalized Gaussian random variables, and we analyze the application of quasi-Monte Carlo (QMC) integration to assess the PDE response statistics using randomly shifted rank-1 lattice rules. In addition to the QMC error analysis, we also consider the dimension truncation and finite element errors in this setting.
\end{abstract}

\footnotetext[2]{Johann Radon Institute for Computational and Applied Mathematics, Austrian Academy of Sciences, Altenbergerstra{\ss}e 69, AT-4040 Linz, Austria ({\tt philipp.guth@ricam.oeaw.ac.at}).}
\footnotetext[3]{School of Engineering Sciences, LUT University, P.O.~Box 20, FI-53851 Lappeenranta, Finland ({\tt vesa.kaarnioja@lut.fi}).}

\section{Introduction}

Uncertainty quantification for partial differential equations (PDEs) with random coefficients has become a crucial aspect of modeling and simulation across various scientific and engineering disciplines. One reason behind this success is the ability to integrate knowledge of governing physical equations while accommodating for randomness which may reflect, for instance, missing data, uncertainties in measurements, material heterogeneity, or external influences. A key focus in this area is the statistical characterization of the solution given the input statistics of the random field. The significant computational expense associated with these problems has spurred research into the development of efficient numerical algorithms, such as those based on sparse grids~\cite{gegr98,nori96,smolyak} or quasi-Monte Carlo (QMC) methods~\cite{dks13,Kuo2016ApplicationOQ,KuoNuyens2018}.

Let $D\subset\mathbb R^d$ be a nonempty, bounded Lipschitz domain, $(\Omega,\mathcal A,\mathbb P)$ a probability space, and $f\!:D\to\mathbb R$ a given source term. An important model problem in applied uncertainty quantification is the elliptic PDE
\begin{align}
\begin{split}
-\nabla\cdot(a(\bsx,\omega)\nabla u(\bsx,\omega))&=f(\bsx), \quad\quad\bsx\in D,\\
u(\bsx,\omega)&=0, \qquad\qquad\!\!\!\bsx\in\partial D
\end{split}\quad\text{for a.e.}~\omega\in\Omega,\label{eq:pdemodel}
\end{align}
where the diffusion coefficient $a\!:D\times \Omega\to \mathbb R$ is assumed to be a lognormal random field. If $\log(a(\bsx,\omega))$ is a Gaussian random field with continuous covariance, then there holds by the Karhunen--Lo\`eve theorem that%
\begin{align}\label{eq:lognormalref}%
a(\bsx,\omega)=a_0(\bsx)\exp\bigg(\sum_{j=1}^\infty \sqrt{\lambda_j}\psi_j(\bsx)y_j(\omega)\bigg)\quad\text{for a.e.}~\bsx\in D~\text{and}~\omega\in \Omega,
\end{align}
where $y_j\overset{\rm i.i.d.}{\sim}\mathcal N(0,1)$ and $(\lambda_j,\psi_j)\in\mathbb R_+\times L^2(D)$, $j\geq 1$, are the eigenpairs of the covariance operator. %
In turn, this allows to identify $a(\bsx,\omega)\equiv a(\bsx,\bsy)$ and $u(\bsx,\omega)\equiv u(\bsx,\bsy)$ as parametric functions of $\bsy=\bsy(\omega)$ and the response statistics of the problem~\eqref{eq:pdemodel} can be recast as infinite-dimensional parametric integrals: for example, the expected value of some quantity of interest $G(u(\cdot,\omega))$ is 
\begin{align*}
\mathbb E[G(u)]=\int_{\Omega}G(u(\cdot,\omega))\,\mathbb P({\rm d}\omega)=\int_{\mathbb R^{\mathbb N}}G(u(\cdot,\bsy))\bsmu_2({\rm d}\bsy),
\end{align*}
where $\bsmu_2=\bigotimes_{j=1}^\infty \mathcal N(0,1)$ denotes the infinite-dimensional product Gaussian measure. 

A popular method for the numerical approximation of the above problem is to represent the input random field as a finite, parametric sum
\begin{align*}
a_s(\bsx,\bsy)=a_0(\bsx)\exp\bigg(\sum_{j=1}^s \sqrt{\lambda_j}\psi_j(\bsx)y_j\bigg),\quad \bsx\in D~\text{and}~\bsy\in\mathbb R^s,
\end{align*}
and it remains to use cubature rules such as QMC methods or sparse grids to approximate high-dimensional integrals
\begin{align*}
I_s(G(u_s))=\int_{\mathbb R^s}G(u_s(\cdot,\bsy_{\leq s}))\bigg(\prod_{j=1}^s \frac{1}{\sqrt{2\pi}}{\rm e}^{-\frac12 y_j^2}\bigg)\,{\rm d}\bsy_{\leq s},
\end{align*}
where $\bsy_{\leq s}:=(y_j)_{j=1}^s$ and the dimensionally-truncated solution $u_s\!: D\times \mathbb R^s\to\mathbb R$ is the solution to the parametric PDE
\begin{align}\label{eq:parametricpde2}
\begin{split}
-\nabla\cdot(a_s(\bsx,\bsy)\nabla u_s(\bsx,\bsy))&=f(\bsx), \quad\quad\bsx\in D,\\
u_s(\bsx,\bsy)&=0, \qquad\quad\quad\!\!\!\bsx\in\partial D
\end{split}\quad\text{for a.e.}~\bsy\in\mathbb R^s.
\end{align}
This approximation incurs a dimension truncation error
\begin{align*}
|\mathbb E[G(u)]-\mathbb E[G(u_s)]|,
\end{align*}
which is controlled by the rate of decay of the eigenvalues $(\lambda_j)_{j\geq 1}$~\cite{matern1,gk24}.

The dimensionally-truncated integrals may still be high-dimensional. Sparse grid methods and QMC methods have proven effective in solving high-dimensional integration problems, particularly in the context of elliptic PDEs with uniform~\cite{dicklegiaschwab,spodpaper14,kss12,kssmultilevel} and lognormal~\cite{log,log2,log3,log4,log5} parameterizations of a random diffusion coefficient. 
There have been recent attempts in the literature to generalize the aforementioned model problem by replacing the probability measure corresponding to the input random field with a generalized $\beta$-Gaussian distribution. 
Specifically, Herrmann {\em et al.}~\cite{hks21} studied QMC integration for Bayesian inverse problems governed by parametric PDEs, where the probability measure was assumed to belong to the family of generalized $\beta$-Gaussian distributions, while Guth and Kaarnioja~\cite{gk24} investigated the dimension truncation error rates subject to high-dimensional integration problems with respect to these probability measures. %

Meanwhile, there has been a recent surge of interest in the study of parametric PDE problems subject to {\em Gevrey regular} input random fields~\cite{chernovle2,chernovle1,schmidlin24}: that is, the parametric coefficient $a_s(\bsx,\bsy)$ in~\eqref{eq:parametricpde2} is obtained as the dimension truncation of a parametric coefficient $a(\bsx,\bsy)$, which is assumed to satisfy, for {\em Gevrey parameter~$\sigma \ge 1$} and some sequence $(b_j)_{j\geq 1}\in\ell^1(\mathbb N)$ of nonnegative numbers,
\begin{align}\label{eq:gevrey0}
\bigg\|\bigg(\prod_{j\geq 1} \frac{\partial^{\nu_j}}{\partial y_j^{\nu_j}}\bigg)a(\cdot,\bsy)\bigg\|_{L^\infty(D)}\lesssim \bigg(\bigg(\sum_{j\geq 1} \nu_j\bigg)!\bigg)^\sigma \prod_{j\geq 1} b_j^{\nu_j}
\end{align}
for all $(\nu_j)_{j\ge1}\in\mathbb N_0^{\mathbb N}$ with $\sum_{j\geq 1}\nu_j<\infty$ and $\bsy=(y_1,y_2,\ldots)\in U$ in some set $\varnothing\neq U\subset \mathbb R^{\mathbb N}$. The papers~\cite{chernovle2,schmidlin24} show that the solutions of semilinear elliptic PDEs are Gevrey regular with respect to the uncertain, bounded parameters if the input random field is Gevrey regular, while~\cite{chernovle1} considered an isolated eigenpair of a linear elliptic PDE subject to the same input parameterization. Furthermore, since Gevrey regular functions are not necessarily complex-analytic, the development of QMC rules for the approximation of the response statistics to such models cannot make use of complex-analytic arguments such as those employed in~\cite{andreev12,bieri09,dglgs19,dicklegiaschwab,schwab_gittelson_2011,ZechPhD}.

A limitation of the aforementioned works~\cite{chernovle2,chernovle1,schmidlin24} is that they operate under the assumption that the parametric regularity bound~\eqref{eq:gevrey0} is uniform over $\bsy\in U$. In particular, this excludes lognormal parameterizations of the input random field $a$. Thus it is the goal of this paper to generalize these results for PDE problems subject to parameter-dependent upper bounds. Our theory also covers the case where the random variables appearing in the input random field are generalized $\beta$-Gaussian random variables. The present work thereby seeks to merge these two concurrent lines of contemporary research of QMC methods involving generalized Gevrey regular parameterizations for input random fields with unbounded support. %

This paper is organized as follows. Subsection~\ref{sec:notations} introduces the multi-index notation used throughout this paper. The modeling assumptions and problem setting are described in Section~\ref{sec:problem}. The main new parametric regularity results are derived in Section~\ref{sec:reg}. Section~\ref{sec:infinitedim} briefly discusses the notion of integrability required within our infinite-dimensional framework while the dimension truncation error analysis for the model problem is carried out in Section~\ref{sec:dimtrunc}. The application of QMC methods for the model problem is considered in Section~\ref{sec:qmcerror}. The QMC error bound is combined with the dimension truncation error rate and finite element error bound in Section~\ref{sec:totalerror}, while numerical experiments demonstrating the QMC error rate are presented in Section~\ref{sec:numex}. The paper ends with some conclusions on our results.

\subsection{Notations and preliminaries}\label{sec:notations}
Throughout this manuscript, boldfaced letters are used to denote multi-indices while the subscript notation $m_j$ is used to refer to the $j^{\rm th}$ component of multi-index $\boldsymbol m$. We denote the set of all finitely supported multi-indices by
\begin{align*}
\mathscr F := \{\boldsymbol m\in\mathbb N_0^{\mathbb N}:|{\rm supp}(\boldsymbol m)|<\infty\},
\end{align*}
where the support of a multi-index is defined as ${\rm supp}(\boldsymbol m):=\{j\in\mathbb N:m_j\neq 0\}$. Moreover, the modulus of a multi-index is defined as
\begin{align*}
|\boldsymbol m|:=\sum_{j\geq 1} m_j.
\end{align*}
Furthermore, for any sequence $\boldsymbol y:=(y_j)_{j=1}^\infty$ of real numbers and $\boldsymbol m,\bsnu\in\mathscr F$, we define the special notations
\begin{align*}
&\boldsymbol m\leq \bsnu\quad\text{if and only if}\quad m_j\leq \nu_j~\text{for all}~j\geq 1,\\
&\delta_{\bsnu,\boldsymbol{m}} = \begin{cases} 1 & \text{if}~\nu_j = m_j~\text{for all}~j\ge 1,\\ 0 & \text{otherwise} ,
\end{cases}\\
&\binom{\bsnu}{\boldsymbol m}:=\prod_{j\geq 1}\binom{\nu_j}{m_j},\quad \boldsymbol m!:=\prod_{j\geq 1}m_j!,\\
&\partial^{\boldsymbol m}:=\prod_{j\geq 1}\frac{\partial^{m_j}}{\partial y_j^{m_j}},\quad \boldsymbol y^{\boldsymbol m}:=\prod_{j\geq 1} y_j^{m_j},
\end{align*}
where we use the convention $0^0:=1$.

We will also use the special notation $\{1:s\}:=\{1,\ldots,s\}$ for $s\in\mathbb N$. Given a set~$\setu \subseteq \{1:s\}$, we write~$-\setu := \{1:s\}\setminus \setu$ and denote by~$\bsy_{\setu}$ the projection of~$\bsy$ onto its components with indices~$j \in \setu$.

\section{Problem setting}\label{sec:problem}
Let $D\subset\mathbb R^d$ be a nonempty, bounded Lipschitz domain with $d\in\{1,2,3\}$. By~$H_0^1(D)$ we denote the subspace of~$H^1(D)$ with zero trace on $\partial D$. We equip $H_0^1(D)$ with the norm~$\|v\|_{H_0^1(D)} = \|\nabla v\|_{L^2(D)}$ and the inner product~$\langle v_1, v_2\rangle_{H_0^1(D)} = \langle \nabla v_1, \nabla v_2 \rangle_{L^2(D)}$. The dual space of~$H_0^1(D)$ is denoted by~$H^{-1}(D)$, $L^2(D)$ is identified with its own dual, and we denote the duality pairing between $H^{-1}(D)$ and $H_0^1(D)$ by $\langle\cdot,\cdot\rangle_{H^{-1}(D),H_0^1(D)}$. Moreover, we denote by $H^t(D)$, $t\in(-1,1)$, the interpolation spaces between $H^{-1}(D)$ and $H^1(D)$, with the convention that $H^0(D)=L^2(D)$. Furthermore, we denote by $C^{0,t}(\overline{D})$ the space of H\"older continuous functions on $\overline{D}$ with exponent $0<t\leq 1$.

Given a functional $f\in H^{-1}(D)$, we consider the model problem~\eqref{eq:pdemodel} in its weak formulation: for all $\bsy\in \bbR^\bbN$, find $u(\cdot,\bsy)\in H_0^1(D)$ such that
\begin{align}\label{eq:weak-rand}
\int_D a(\bsx,\bsy)\nabla u(\bsx,\bsy)\cdot \nabla v(\bsx)\,{\rm d}\bsx=\langle f,v\rangle_{H^{-1}(D),H_0^1(D)}\quad\text{for all}~v\in H_0^1(D).
\end{align}
We shall then be interested in integrals involving the solution of~\eqref{eq:weak-rand}, such as
\begin{align*}
    \int_{\bbR^\bbN} u(\cdot,\bsy) \bsmu_\beta(\mathrm d\bsy) \in H_0^1(D) \quad \text{or} \quad \int_{\bbR^\bbN} G(u(\cdot,\bsy)) \bsmu_\beta(\mathrm d\bsy) \in \bbR.
\end{align*}
Here and in the following,~$G \in H^{-1}(D)$ denotes a~$\bsy$-independent, bounded and linear functional on~$H_0^1(D)$, and~$\bsmu_\beta$ is the generalized~$\beta$-Gaussian distribution defined as
\begin{align}
    \bsmu_\beta := \bigotimes_{j\geq1} \mathcal{N}_\beta(0,1),\label{eq:gengauss}
\end{align}
where $\mathcal{N}_\beta(0,1)$ denotes the univariate $\beta$-Gaussian distribution with density 
\begin{align}
    \varphi_{\beta}(y) := \frac{1}{2\beta^{\frac{1}{\beta}}\Gamma(1+\frac{1}{\beta})} {\rm e}^{-\frac{|y|^\beta}{\beta}}, \quad y\in \bbR,\label{eq:betadef}
\end{align}
where $\beta >0$. Importantly, in the case $\beta=2$ the probability measure~\eqref{eq:gengauss} is Gaussian and in the case $\beta=1$ it corresponds to the Laplace distribution. Formally, the case $\beta=\infty$ corresponds to the uniform probability measure on $[-1,1]^{\mathbb N}$.%

Given~$0< \tau \leq \beta$, and an arbitrary sequence~$\bsalpha:=(\alpha_j)_{j\geq 1}\in\ell^1(\mathbb N)$, such that~$\alpha_j\in[0,\infty)$ for all $j\in\mathbb N$, we define the set
\begin{align*}
    U_{\bsalpha,\tau} := \bigg\{ \bsy \in \bbR^{\bbN}: \sum_{j\geq 1} \alpha_j |y_j|^{\tau} <\infty\bigg\}.
\end{align*}
In Lemma~\ref{lem:domain_change} it is shown that~$\bsmu_\beta({U_{\bsalpha,\tau}}) = 1$. Thus, in the $\beta$-Gaussian setting, the domain of integration $\mathbb R^{\mathbb N}$ is interchangeable with $U_{\bsalpha,\tau}$, and we restrict ourselves to
\begin{align*}
    &I(u) := \int_{U_{\bsalpha,\tau}} u(\bsy) \bsmu_\beta(\mathrm d\bsy) \in H_0^1(D)\\
    &\text{or}\quad G(I(u)) = I(G(u)) = \int_{U_{\bsalpha,\tau}} G(u(\bsy)) \bsmu_\beta(\mathrm d\bsy) \in \bbR.
\end{align*}

The numerical evaluation of these integrals requires different types of approximations. First, the infinite-dimensional integrals~$I$ are truncated to $s$-dimensional ones~$I_s$. To this end, we denote the {\em dimensionally-truncated} diffusion coefficient by
\begin{align*}a_s(\cdot,\bsy):=a(\cdot,(y_1,\ldots,y_s,0,0,\ldots))\quad \text{for $\bsy\in\mathbb R^{\mathbb N}$}\end{align*} and the {\em dimensionally-truncated} PDE solution by \begin{align*}u_s(\cdot,\bsy):=u(\cdot,(y_1,\ldots,y_s,0,0,\ldots))\quad \text{for $\bsy\in\mathbb R^{\mathbb N}$.}\end{align*} Then, an~$n$-point cubature rule~$Q_{s,n}$, e.g., a QMC rule, can be applied to approximate the~$s$-dimensional integral over a spatial discretization~$u_h$ of the PDE solution~$u$. The total error can be decomposed as
\begin{align}\label{eq:overallerror}\begin{split}
&\|I(u)- Q_{s,n}(u_h)\|_{H_0^1(D)}\\
&\le \|(I-I_s)(u)\|_{H_0^1(D)} + \|I_s(u-u_h)\|_{H_0^1(D)} + \|(I_s- Q_{s,n})(u_h)\|_{H_0^1(D)}.
\end{split}
\end{align}
In order to control these different error contributions and to derive rigorous convergence rates we will work under the following assumptions.

\begin{enumerate}[label=(A\arabic*)]
\item Let $\boldsymbol b=(b_j)_{j\geq 1}$ be a sequence of nonnegative real numbers. We assume that the function $a\!:D\times U_{\bsalpha,\tau}\to \mathbb R$ satisfies for some $\sigma\geq 1$ that
\begin{align*}
\bigg\|\frac{\partial^{\bsnu}a(\cdot,\bsy)}{a(\cdot,\bsy)}\bigg\|_{L^\infty(D)} &\leq C(|\bsnu|!)^\sigma \boldsymbol b^{\bsnu}
\end{align*}
for all $\bsy\in U_{\bsalpha,\tau}$ and $\bsnu\in\mathscr F$, where~$C>0$ is a constant independent of $\bsy$ and $\bsnu$.\label{assump2}
\item We assume that there holds $\|a(\cdot,\bsy)-a_s(\cdot,\bsy)\|_{L^\infty(D)}\xrightarrow{s\to\infty}0$%
~for all $\bsy\in U_{\bsalpha,\tau}$.\label{assumption:prestrang}
\item The sequence~$\bsb = (b_j)_{j\ge1}$ is~$p$-summable, i.e., there exists $p\in(0,1)$ such that $\bsb\in\ell^p(\bbN)$.\label{assump3}
\item The sequence~$\bsb = (b_j)_{j\ge1}$ is monotonically decreasing, i.e., $b_1\geq b_2\geq \cdots$.\label{assump4}
\item There holds
\begin{align*}
a(\bsx,\bsy)\geq a_{\min}(\bsy) :=  c \exp\bigg(-\sum_{j=1}^\infty \alpha_j|y_j|^{\tau}\bigg)
\end{align*}
for all $\bsx\in D$ and $\bsy\in U_{\bsalpha,\tau}$, where~$c>0$,~$0<\tau \le\beta$, and~$\bsalpha \in \ell^1(\bbN)$. If $\tau=\beta$, then we additionally require $\|\bsalpha\|_\infty:=\sup_{j\geq 1}|\alpha_j|<\beta^{-1}$.\label{assump5}
\item Let~$a \in L_{\bsmu_\beta}^q(U_{\bsalpha,\tau},C^{0,t}(\overline{D}))$ and~$f \in %
H^{t-1}(D)$, for some~$0<t\le 1$ and for all~$q \in (0,\infty)$.\label{assump7}%
\item The spatial domain $D$ is a convex polyhedron.\label{assump8}
\end{enumerate}
{\em Remarks:}
\begin{itemize}
\item[(i)] The Assumptions~\ref{assump2}--\ref{assump7} are satisfied, in particular, by the lognormally parameterized coefficient~\eqref{eq:lognormalref} discussed in the introduction with $\beta=2$ and $\sigma=1$ (see, e.g.,~\cite{log,Kuo2016ApplicationOQ}). However, these assumptions cover a more general setting.
\item[(ii)]Assumption~\ref{assump7} implies that
\begin{align*}
    a_{\rm max}(\bsy) :=& \|a(\bsy)\|_{C^{0,t}(\overline{D})} \\
    =&\,\|a(\cdot,\bsy)\|_{L^\infty(\overline{D})}+\sup_{{\bsx,\bsx'\in \overline{D},~\bsx\neq\bsx'}}\frac{|a(\bsx,\bsy)-a(\bsx',\bsy)|}{\|\bsx-\bsx'\|^t}\in L_{\bsmu_\beta}^q(U_{\bsalpha,\tau})
\end{align*}
for all $q\in (0,\infty)$.
\item[(iii)]Assumption~\ref{assumption:prestrang} together with Strang's second lemma (see, e.g.,~\cite{matern1}) implies that
\begin{align*}
\|u(\cdot,\bsy)-u_s(\cdot,\bsy)\|_{H_0^1(D)}\xrightarrow{s\to\infty}0\quad\text{for all}~\bsy\in U_{\bsalpha,\tau}.
\end{align*}
\end{itemize}

\section{Regularity analysis of the model problem}\label{sec:reg}

In this section, we derive a parametric regularity bound for the PDE solution under the assumptions%
~\ref{assump2} and~\ref{assump5}. In Subsection~\ref{sec:abstract}, we prove an abstract multivariate recurrence bound. This abstract result will be used to derive the main regularity bound for the PDE solution in Subsection~\ref{sec:mainreg}.

\subsection{Abstract multivariate recurrence bound}\label{sec:abstract}

\begin{lemma}\label{lemma:easy}
Let $(\Upsilon_{\bsnu})_{\bsnu\in\mathscr F}$ and $\bsb=(b_j)_{j\geq 1}$ be sequences of nonnegative real numbers satisfying
\begin{align}\label{eq:upsilon}
\Upsilon_{\mathbf 0}\leq C_0\quad\text{and}\quad \Upsilon_{\bsnu}\leq C\sum_{\substack{\boldsymbol m\leq \bsnu\\ \boldsymbol m\neq \mathbf 0}}\binom{\bsnu}{\boldsymbol m}(|\boldsymbol m|!)^\sigma \boldsymbol b^{\boldsymbol m}\Upsilon_{\bsnu-\boldsymbol m}\quad\text{for all}~\bsnu\in\mathscr F\setminus\{\mathbf 0\},
\end{align}
where $C_0,C>0$ and $\sigma\geq 1$. 
Then there holds
\begin{align}\label{eq:upsilon2}
\Upsilon_{\bsnu}\leq C_0a_{|\bsnu|}(|\bsnu|!)^\sigma \boldsymbol b^{\bsnu},
\end{align}
where
\begin{align*}
a_k=C^{1-\delta_{k,0}}(C+1)^{\max\{k-1,0\}}.
\end{align*}
In the special case $\sigma=1$ this result is sharp in the sense that if equality holds in~\eqref{eq:upsilon}, then~\eqref{eq:upsilon2} also holds with equality.
\end{lemma}
\proof We prove this claim by induction with respect to the order of the multi-index $\bsnu\in\mathscr F$. The basis of the induction $\bsnu=\mathbf 0$ follows immediately from the assumptions. We let $\bsnu\in\mathscr F\setminus\{\mathbf 0\}$ and assume that the claim has been proved for all multi-indices with order less than $|\bsnu|$. Then
\begin{align*}
\Upsilon_{\bsnu}&\leq C_0C\!\sum_{\substack{\boldsymbol m\leq \bsnu\\ \boldsymbol m\neq\mathbf 0}}\!\binom{\bsnu}{\boldsymbol m}\!(|\boldsymbol m|!)^\sigma \boldsymbol b^{\boldsymbol m}C^{1-\delta_{|\bsnu|\!-\!|\boldsymbol m|,0}}(C+1)^{\max\{|\bsnu|\!-\!|\boldsymbol m|\!-\!1,0\}}((|\boldsymbol\nu|\!-\!|\boldsymbol m|)!)^\sigma \boldsymbol b^{\bsnu-\boldsymbol m}\\
&=C_0C(|\bsnu|!)^\sigma \boldsymbol b^{\bsnu}+C_0C\boldsymbol b^{\bsnu}\sum_{\substack{\boldsymbol m\leq \bsnu\\ \boldsymbol m\neq \mathbf 0\\ \boldsymbol m\neq \bsnu}}\binom{\bsnu}{\boldsymbol m}C(C+1)^{|\bsnu|-|\boldsymbol m|-1}(|\boldsymbol m|!(|\bsnu|-|\boldsymbol m|)!)^\sigma\\
&=C_0C(|\bsnu|!)^\sigma \boldsymbol b^{\bsnu}-C_0\boldsymbol b^{\bsnu}C^2(C+1)^{|\bsnu|-1}(|\bsnu|!)^\sigma -C_0\boldsymbol b^{\bsnu}C^2(C+1)^{-1}(|\bsnu|!)^\sigma\\
&\quad +C_0\boldsymbol b^{\bsnu}\sum_{\boldsymbol m\leq \bsnu}\binom{\bsnu}{\boldsymbol m}C^2(C+1)^{|\bsnu|-|\boldsymbol m|-1}(|\boldsymbol m|!(|\bsnu|-|\boldsymbol m|)!)^\sigma\\
&=C_0C(|\bsnu|!)^\sigma \boldsymbol b^{\bsnu}-C_0\boldsymbol b^{\bsnu}C^2(C+1)^{|\bsnu|-1}(|\bsnu|!)^\sigma-C_0\boldsymbol b^{\bsnu}C^2(C+1)^{-1}(|\bsnu|!)^\sigma\\
&\quad +C_0C^2(C+1)^{|\bsnu|-1}\boldsymbol b^{\bsnu}\sum_{\ell=0}^{|\bsnu|}(C+1)^{-\ell}(\ell!(|\bsnu|-\ell)!)^\sigma\sum_{\substack{\boldsymbol m\leq \bsnu\\ |\boldsymbol m|=\ell}}\binom{\bsnu}{\boldsymbol m}\\
&=C_0C(|\bsnu|!)^\sigma \boldsymbol b^{\bsnu}-C_0\boldsymbol b^{\bsnu}C^2(C+1)^{|\bsnu|-1}(|\bsnu|!)^\sigma -C_0\boldsymbol b^{\bsnu}C^2(C+1)^{-1}(|\bsnu|!)^\sigma\\
&\quad +C_0C^2(C+1)^{|\bsnu|-1}\boldsymbol b^{\bsnu}\sum_{\ell=0}^{|\bsnu|}(C+1)^{-\ell}(\ell!(|\bsnu|-\ell)!)^\sigma\frac{|\bsnu|!}{\ell!(|\bsnu|-\ell)!},
\intertext{where we used the Vandermonde convolution identity $\sum_{\boldsymbol m\leq\bsnu,~|\boldsymbol m|=\ell}\binom{\bsnu}{\boldsymbol m}=\binom{|\bsnu|}{\ell}=\frac{|\bsnu|!}{(|\bsnu|-\ell)!\ell!}$ (see, e.g.,~\cite[Equation~(5.1)]{gouldbook}). Now it follows that}
\Upsilon_{\bsnu}&\leq C_0C(|\bsnu|!)^\sigma \boldsymbol b^{\bsnu}-C_0\boldsymbol b^{\bsnu}C^2(C+1)^{|\bsnu|-1}(|\bsnu|!)^\sigma -C_0\boldsymbol b^{\bsnu}C^2(C+1)^{-1}(|\bsnu|!)^\sigma\\
&\quad +C_0C^2(C+1)^{|\bsnu|-1}\boldsymbol b^{\bsnu}(|\bsnu|!)^\sigma\sum_{\ell=0}^{|\bsnu|}(C+1)^{-\ell}\\
&=C_0C(|\bsnu|!)^\sigma \boldsymbol b^{\bsnu}-C_0\boldsymbol b^{\bsnu}C^2(C+1)^{|\bsnu|-1}(|\bsnu|!)^\sigma-C_0\boldsymbol b^{\bsnu}C^2(C+1)^{-1}(|\bsnu|!)^\sigma\\
&\quad +C_0C(C+1)^{|\bsnu|-1}\boldsymbol b^{\bsnu}(|\bsnu|!)^\sigma(1+C-(C+1)^{-|\bsnu|})\\
&=(|\bsnu|!)^\sigma\boldsymbol b^{\bsnu}\big(C_0C-C_0C^2(C+1)^{|\bsnu|-1}-C_0C^2(C+1)^{-1}\\
&\quad\quad\quad\quad\quad\quad+C_0C(C+1)^{|\bsnu|-1}(1+C-(C+1)^{-|\bsnu|})\big)\\
&=(|\bsnu|!)^\sigma\boldsymbol b^{\bsnu} C_0C (C+1)^{|\bsnu|-1},
\end{align*}
where we used the fact that $(\ell!(|\bsnu|-\ell)!)^{\sigma-1}\leq (|\bsnu|!)^{\sigma-1}$ for all $\sigma\geq 1$.\quad\endproof

\subsection{Parametric regularity bound for the PDE solution}\label{sec:mainreg}

\begin{theorem}\label{thm:weak-rand}
    Under the assumptions%
    ~{\rm \ref{assump2}} and~{\rm \ref{assump5}} the solution of~\eqref{eq:weak-rand} is Gevrey regular with the same parameter~$\sigma\geq 1$, i.e., for all~$\bsnu \in \mathscr F\setminus\{\mathbf 0\}$ and $\bsy\in U_{\bsalpha,\tau}$, there holds 
        \begin{align}\label{eq:reg2}
        \| \partial^{\bsnu}u(\cdot,\bsy)\|_{H_0^1(D)}%
        \leq \frac{\|f\|_{H^{-1}(D)}}{a_{\min}(\bsy)}C(C+1)^{|\bsnu|-1}(|\bsnu|!)^\sigma\boldsymbol b^{\bsnu},
    \end{align}
    where $C>0$ is the constant in~{\rm \ref{assump2}}.
\end{theorem}

\proof
Let $\bsnu\in\mathscr F\setminus\{\mathbf 0\}$. Following the argument in~\cite[Theorem~2.1]{beck}, we obtain
\begin{align*}
&\|a(\cdot,\bsy)^{1/2}\nabla \partial^{\bsnu}u(\cdot,\bsy)\|_{L^2(D)}^2\\
&\leq \sum_{\substack{\boldsymbol m\leq \bsnu\\ \boldsymbol m\neq\mathbf 0}}\binom{\bsnu}{\boldsymbol m}\int_D \bigg|\frac{\partial^{\boldsymbol m}a(\bsx,\bsy)}{a(\bsx,\bsy)}a(\bsx,\bsy)\nabla \partial^{\bsnu-\boldsymbol m}u(\bsx,\bsy)\cdot \nabla \partial^{\bsnu}u(\bsx,\bsy)\bigg|\,{\rm d}\bsx\\
&\leq C\sum_{\substack{\boldsymbol m\leq \bsnu\\ \boldsymbol m\neq\mathbf 0}}\binom{\bsnu}{\boldsymbol m}(|\boldsymbol m|!)^\sigma\boldsymbol b^{\boldsymbol m}\|a(\cdot,\bsy)^{1/2}\nabla \partial^{\bsnu-\boldsymbol m}u(\cdot,\bsy)\|_{L^2(D)}\|a(\cdot,\bsy)^{1/2}\nabla \partial^{\bsnu}u(\cdot,\bsy)\|_{L^2(D)},
\end{align*}
which yields the recurrence relation
\begin{align*}
\|a(\cdot,\bsy)^{1/2}\nabla \partial^{\bsnu}u(\cdot,\bsy)\|_{L^2(D)}&\leq C\sum_{\substack{\boldsymbol m\leq \bsnu\\ \boldsymbol m\neq\mathbf 0}}\binom{\bsnu}{\boldsymbol m}(|\boldsymbol m|!)^\sigma\boldsymbol b^{\boldsymbol m}\|a(\cdot,\bsy)^{1/2}\nabla \partial^{\bsnu-\boldsymbol m}u(\cdot,\bsy)\|_{L^2(D)}.%
\end{align*}
One can apply Lemma~\ref{lemma:easy} together with the {\em a priori} bound \begin{align*}\|a(\cdot,\bsy)^{1/2}\nabla u(\cdot,\bsy)\|_{L^2(D)}\leq \frac{\|f\|_{H^{-1}(D)}}{a_{\min}(\bsy)^{1/2}}\quad\text{for all}~\bsy\in U_{\bsalpha,\tau},\end{align*}
to obtain (see, e.g., \cite[Proof of Lemma~6.5]{Kuo2016ApplicationOQ})
\begin{align*}
\|a(\cdot,\bsy)^{1/2}\nabla \partial^{\bsnu}u(\cdot,\bsy)\|_{L^2(D)}\leq \frac{\|f\|_{H^{-1}(D)}}{a_{\min}(\bsy)^{1/2}}C(C+1)^{|\bsnu|-1}(|\bsnu|!)^\sigma\boldsymbol b^{\bsnu}\,\,\text{for all}~\bsnu\in\mathscr F\setminus\{\mathbf 0\},
\end{align*}
which together with~\ref{assump5} proves~\eqref{eq:reg2}.\quad\endproof

\section{Note on infinite-dimensional integration}\label{sec:infinitedim}

The following lemma, which is adapted from~\cite[Lemma~2.28]{schwab_gittelson_2011} to the present setting, allows us to interchange the domain of integration~$\bbR^\bbN$ with~$U_{\bsalpha,\tau}$.
\begin{lemma}\label{lem:domain_change}
    There holds~$U_{\bsalpha,\tau}\in \mathcal{B}(\bbR^{\mathbb N})$, where~$\mathcal{B}(\bbR^{\mathbb N})$ denotes the Borel~$\sigma$-algebra generated by~$\bsmu_{\beta}$ and the Borel cylinders in~$\bbR^\bbN$. Moreover,~$\bsmu_\beta(U_{\bsalpha,\tau}) =1$.
\end{lemma}
\begin{proof}
The first statement follows from 
\begin{align*}
    U_{\bsalpha,\tau} = \bigcup_{N\ge 1} \bigcap_{M \ge 1} \left\{ \bsy \in \bbR^\bbN : \sum_{1\le j \le M} \alpha_j |y_j|^\tau \le N \right\}.
\end{align*}
By the monotone convergence theorem, we obtain
\begin{align*}
    \int_{\bbR^\bbN} \sum_{j \ge 1} \alpha_j |y_j|^\tau \,\bsmu_\beta(\mathrm d \bsy) = \sum_{j \ge1} \alpha_j \int_{\bbR^\bbN} |y_j|^\tau \,\bsmu_\beta(\mathrm d \bsy) = \frac{\Gamma\big(\frac{\tau+1}{\beta}\big)}{\beta^{1-\frac{\tau}{\beta}}\Gamma(1+\frac{1}{\beta})} \sum_{j\ge 1} \alpha_j < \infty
\end{align*}
for all~$\tau,\beta>0$, where we used~\cite[formula 3.326.2]{gradshteynryzhik}. Since the sum converges $\boldsymbol\mu_{\beta}$-a.e.~on $\mathbb R^{\mathbb N}$, it follows that $U_{\boldsymbol\alpha,\tau}$ is of measure one.\quad
\end{proof}

We will next show that~$\bsy \mapsto u(\cdot,\bsy)$ is~$\bsmu_\beta$-integrable. Indeed, from Theorem~\ref{thm:weak-rand} we infer that~$\bsy \mapsto G(u(\cdot,\bsy))$ for all~$G \in H^{-1}(D)$ is continuous as a composition of continuous mappings. Thus,~$\bsy \mapsto G(u(\cdot,\bsy))$ is measurable for all~$G \in H^{-1}(D)$, i.e.,~$\bsy \mapsto u(\cdot,\bsy)$ is weakly measurable. By Pettis' theorem (cf., e.g.,~\cite[Chapter~4]{Yosida}) we obtain that~$\bsy \mapsto u(\cdot,\bsy)$ is strongly measurable.
The~$\bsmu_{\beta}$-integrability of the upper bound in Theorem~\ref{thm:weak-rand} follows from the proof of Proposition~\ref{prop:amin} unconditionally if $\beta>\tau$; on the other hand, if $\beta=\tau$, then we need to additionally require that $\|\bsalpha\|_{\infty}<\frac{1}{\beta}$ to ensure the $\bsmu_{\beta}$-integrability of the upper bound in Theorem~\ref{thm:weak-rand}. Under these conditions, by Bochner's theorem (cf., e.g.,~\cite[Chapter~5]{Yosida}) we conclude that~$u$ is $\bsmu_{\beta}$-integrable over~$U_{\bsalpha,\tau}$.

\begin{proposition}\label{prop:amin}
    Let assumption~{\rm \ref{assump5}} hold. In the case~$\tau < \beta$ we have~$\frac{1}{a_{\min}} \in L^q_{\bsmu_\beta}(U_{\bsalpha,\tau})$ for any~$q \in (0,\infty)$. Otherwise, if~$\tau = \beta$ we have~$\frac{1}{a_{\min}} \in L^q_{\bsmu_\beta }(U_{\bsalpha,\tau})$ for all~$q \in (0,\frac{1}{\beta \|\bsalpha\|_{\infty}})$.
\end{proposition}
\proof
Recall that
    \begin{align}
        \left\|\frac{1}{a_{\min}}\right\|_{L^q_{\mu_\beta}(U_{\bsalpha,\tau})}^q &= \int_{U_{\bsalpha,\tau}} \bigg|\frac{1}{a_{\min}(\bsy)}\bigg|^q\,\boldsymbol \mu_\beta(\mathrm d\bsy)\notag\\
        &=  c^{-q} \int_{U_\bsalpha,\tau}\exp\bigg(q\sum_{j=1}^\infty \alpha_j|y_j|^{\tau}\bigg)\,\boldsymbol \mu_\beta(\mathrm d\bsy) \notag\\
        &= c^{-q} \prod_{j \ge 1} c_\beta \int_{\mathbb{R}}\exp\bigg(q \alpha_j|y_j|^{\tau} - \frac{1}{\beta} |y_j|^\beta \bigg)\,\mathrm dy_j,\notag
    \end{align}
        where~$c_\beta := \frac{1}{2\beta^{\frac{1}{\beta}}\Gamma(1+\frac{1}{\beta})}$. Thus, if~$\tau = \beta$, we have~$\frac{1}{a_{\min}} \in L^q_{\bsmu_\beta}(U_{\bsalpha,\tau})$ for~$0 < q < \frac{1}{\beta\|\bsalpha\|_{\infty}}$. Now, let~$\tau < \beta$. Since~$\bsalpha \in \ell^1(\bbN)$, there exists~$j' \in \bbN$ such that~$\tau q \alpha_j <\frac{1}{2}$ for all~$j > j'$. Further, 
    \begin{align}
        \left\|\frac{1}{a_{\min}}\right\|_{L^q_{\mu_\beta}(U_{\bsalpha,\tau})}^q  &= C_{j'} \prod_{j \ge j'} c_\beta \int_{\mathbb{R}}\exp\bigg(q \alpha_j|y_j|^{\tau} - \frac{1}{\beta} |y_j|^\beta \bigg)\,\mathrm dy_j,\notag
    \end{align}
    where~$C_{j'} := c^{-q}\prod_{j=1}^{j'} \int_{\bbR} c_\beta \exp(q \alpha_j |y_j|^\tau - \beta^{-1} |y_j|^\beta) \,\mathrm dy_j$. For~$\kappa>1$ and~$x,y\ge 0$ there holds~$xy \le \frac{\kappa-1}{\kappa} x + \frac{1}{\kappa}xy^\kappa$, and thus~$q\alpha_j |y_j|^\tau \le \frac{\beta - \tau }{\beta}q \alpha_j  + \frac{\tau}{\beta} q \alpha_j |y|^{\beta}$. Hence, 
    \begin{align}
        &c_\beta \int_{\mathbb{R}}\exp\bigg(q \alpha_j|y_j|^{\tau} - \frac{1}{\beta} |y_j|^\beta \bigg)\,\mathrm dy_j \notag\\
        &\leq c_\beta\exp\left(\frac{\beta-\tau}{\beta}q \alpha_j  \right) \int_{\bbR} \exp\left(-\frac{1-\tau q \alpha_j}{\beta} |y_j|^\beta\right)\mathrm dy_j \notag\\
        &= \exp\left(\frac{\beta -\tau}{\beta}q \alpha_j  \right) \frac{1}{(1-\tau q \alpha_j)^\frac{1}{\beta}}.\notag
    \end{align}
    Further, since~$(1-x)^{-1} \le \exp(x(1-x)^{-1})$ for all~$x \in [0,1)$, there holds
    \begin{align}
        c_\beta \int_{\mathbb{R}}\exp\bigg(q \alpha_j|y_j|^{\tau} - \frac{1}{\beta} |y_j|^\beta \bigg)\,\mathrm dy_j &\le \exp\left(\frac{\beta -\tau}{\beta}q\alpha_j  \right) \exp\left(\frac{1}{\beta} \frac{\tau q \alpha_j}{1 - \tau q \alpha_j} \right) \notag\\&\le \exp\left( \left(1+\frac{\tau}{\beta}\right)q\alpha_j\right),\notag
    \end{align}
    where we used that~$\tau q \alpha_j < \frac{1}{2}$ for~$j \ge j'$. Finally, we have
    \begin{align}
        \left\|\frac{1}{a_{\min}}\right\|_{L^q_{\mu_\beta}(U_{\bsalpha,\tau})}^q  &= C_{j'} \prod_{j \ge j'} c_\beta \int_{\mathbb{R}}\exp\bigg(r \alpha_j|y_j|^{\tau} - \frac{1}{\beta} |y_j|^\beta \bigg)\,\mathrm dy_j\notag\\
        &\le C_{j'} \prod_{j \ge j'}\exp\left( \left(1+\frac{\tau}{\beta}\right)q\alpha_j\right) \notag\\
        &\le C_{j'} \exp\left( 2q\sum_{j\ge 1} \alpha_j \right) < \infty, \notag
    \end{align}
    since~$\bsalpha \in \ell^1(\bbN)$.\quad\endproof

\section{Dimension truncation error}\label{sec:dimtrunc}
Results on the dimension truncation error for PDEs (or quantities of interest thereof) subject to Gevrey regular input random fields have been derived in~\cite{chernovle2, chernovle1} for parameters with bounded support. The rate~$s^{-2(\frac{1}{p}-1)}$ obtained in these works can be improved to~$s^{-2(\frac{1}{p}-\frac{1}{2})}$ by an application of~\cite[Theorem~4.2]{gk24}. The improved rate is also obtained in the~$\beta$-Gaussian setting with~$\tau = 1$, see~\cite[Theorem~4.1]{gk24}. We generalize this result to the case~$0< \tau \le \beta$, which can be shown by following essentially the same strategy as in the proof of~\cite[Theorem~4.1]{gk24}. For completeness we state the generalized result for~$0< \tau \le \beta$ together with a compact version of the proof that highlights the differences compared to the proof of~\cite[Theorem~4.1]{gk24}.

\begin{theorem}\label{thm:dimtrunc} Let $u(\cdot,\bsy) \in H_0^1(D)$, $\bsy\in U_{\bsalpha,\tau}$, be the solution of~\eqref{eq:weak-rand}. Suppose that assumptions {\rm\ref{assump2}}--{\rm\ref{assump5}} hold. Then
\begin{align*}
\bigg\|\int_{\mathbb R^{\mathbb N}}(u(\cdot,\bsy)-u_s(\cdot,\bsy))\bsmu_\beta({\mathrm d}\bsy)\bigg\|_{H_0^1(D)}\leq Cs^{-\frac{2}{p}+1},
\end{align*}
where the constant $C>0$ is independent of the dimension $s$.

Let $G\in H^{-1}(D)$ be arbitrary. Then
\begin{align*}
\bigg|\int_{\mathbb R^{\mathbb N}}G(u(\cdot,\bsy)-u_s(\cdot,\bsy))\bsmu_\beta({\mathrm d}\bsy)\bigg|\leq C\|G\|_{X'}s^{-\frac{2}{p}+1},
\end{align*}
where the constant $C>0$ is as above.
\end{theorem}
\proof Let $G\in H^{-1}(D)$ be arbitrary and define
\begin{align*}
F_G(\bsy):=\langle G,u(\cdot,\bsy)\rangle_{H^{-1}(D),H_0^1(D)}\quad\text{for all}~\bsy\in U_{\boldsymbol\alpha,\tau}.
\end{align*}
Then
\begin{align*}
\partial^{\bsnu}F_G(\bsy)=\langle G,\partial^{\bsnu}u(\cdot,\bsy)\rangle_{H^{-1}(D),H_0^1(D)}\quad\text{for all}~\bsnu\in\mathscr F~\text{and}~\bsy\in U_{\boldsymbol\alpha,\tau}
\end{align*}
and it follows from our assumptions that
\begin{align*}
|\partial^{\bsnu}F_G(\bsy)|\leq \|G\|_{H^{-1}(D)} \Theta_{|\bsnu|} \boldsymbol b^{\bsnu} \prod_{j\geq 1}{\rm e}^{\alpha_j|y_j|^\tau} \quad\text{for all}~\bsnu\in\mathscr F~\text{and}~\bsy\in U_{\boldsymbol\alpha,\tau},
\end{align*}
where~$\Theta_{|\bsnu|} = \frac{\|f\|_{H^{-1}(D)}}{c} (C+1)^{|\bsnu|}(|\bsnu|!)^\sigma$.
This bound differs from the bound on~$|\partial^\bsnu F_G|$ in the proof of~\cite[Theorem~4.1]{gk24} only by the exponent~$\tau$. For this reason the proof follows essentially the same steps and we will here only highlight the differences.

Let us denote $\bsy_{\leq s}:=(y_j)_{j=1}^s$ and $\bsy_{>s}:=(y_j)_{j=s+1}^\infty$. Developing the Taylor expansion of~$F_G$ about the point~$(\bsy_{\le s},\boldsymbol{0})$ and computing the integral with respect to~$\boldsymbol\mu_{\beta}$ leads to the bound
\begin{align}
&\bigg|\int_{\mathbb R^{\mathbb N}}(F_G(\bsy)-F_G(\bsy_{\leq s},\mathbf 0))\,\bsmu_{\beta}({\mathrm d}\bsy)\bigg|\notag\\
&\leq \sum_{\ell=2}^k \sum_{\substack{|\bsnu|=\ell\\ \nu_j=0~\forall j\leq s\\ \nu_j\neq 1~\forall j>s}}\frac{1}{\bsnu!}\int_{\mathbb R^{\mathbb N}}|\bsy^{\bsnu}|\cdot|\partial^{\bsnu}F_G(\bsy_{\leq s},\mathbf 0)|\,\bsmu_{\beta}({\rm d}\bsy)\label{eq:term1}\\
&\quad +\sum_{\substack{|\bsnu|=k+1\\ \nu_j=0~\forall j\leq s}}\frac{k+1}{\bsnu!}\int_{\mathbb R^{\mathbb N}}\int_0^1 (1-t)^k|\bsy^{\bsnu}|\cdot|\partial^{\bsnu}F_G(\bsy_{\leq s},t \bsy_{>s})|\,{\rm d}t\,\bsmu_\beta({\mathrm d}\bsy),\label{eq:term2}
\end{align}
with~$k = \lceil \frac{1}{1-p} \rceil$, see~\cite[Equations~(4.1)--(4.2)]{gk24}. We begin our estimation by splitting the terms in \eqref{eq:term1}:
\begin{align*}
&\int_{\mathbb R^{\mathbb N}}|\bsy^{\bsnu}|\cdot|\partial^{\bsnu}F_G(\bsy_{\leq s},\mathbf 0)|\,\bsmu_{\beta}({\rm d}\bsy) %
\leq \|G\|_{H^{-1}(D)}\Theta_{|\bsnu|}\bsb^{\bsnu} \int_{\mathbb R^{\mathbb N}}|\bsy^{\bsnu}| \bigg(\prod_{j\geq 1} {\rm e}^{\alpha_j |y_j|^\tau}\bigg)
\bsmu_{\beta}(\mathrm d\bsy)\\
&=\|G\|_{H^{-1}(D)}\Theta_{|\bsnu|}\bsb^{\bsnu} \bigg(\underbrace{\prod_{j\in {\rm supp}{(\bsnu)}} \int_{\mathbb{R}} |y_j|^{\nu_j} {\rm e}^{\alpha_j|y_j|^\tau}\varphi_\beta(y_j){\mathrm d}y_j\!}_{=:\rm term_1}\bigg)\\
&\quad\quad\quad\quad\quad\quad\quad\quad\quad\quad \times\bigg(\underbrace{\prod_{j\notin {\rm supp}{(\bsnu)}}  \int_{\mathbb{R}} {\rm e}^{\alpha_j|y_j|^\tau}\varphi_\beta(y_j){\mathrm d}y_j\!}_{=:\rm term_2}\bigg). 
\end{align*}
In order to bound ${\rm term_1}$, observe that
\begin{align}\label{eq:Cdef}
C_{\alpha_j,\beta,\nu_j}:=\int_{\mathbb R}|y_j|^{\nu_j}{\rm e}^{\alpha_j|y_j|^\tau}\varphi_{\beta}(y_j)\,{\mathrm d}y_j<\infty
\end{align}
since we assumed that $\beta\geq\tau$ and $\alpha_j<\frac{1}{\beta}$ in the case $\beta =\tau$. By defining the auxiliary constant~$A_{\alpha_j,\beta,\ell}:=\max_{1\leq k\leq \ell}C_{\alpha_j,\beta,k}$, we have
\begin{align*}
{\rm term_1} &\leq \max{\{1,A_{\|\bsalpha\|_{\infty},\beta,|\bsnu|}\}}^{|\bsnu|},
\end{align*}
where $\|\bsalpha\|_\infty:=\sup_{j\geq 1}|\alpha_j|$ is finite since $\bsalpha\in\ell^1(\mathbb N)$ by assumption. To bound ${\rm term_2}$, we note that there is an index $j'\in\mathbb N$ such that~$\alpha_j\leq\frac{1}{2\tau}$ for all $j>j'$. Thus
\begin{align*}
{\rm term}_2%
&\leq \max\{1,C_{\|\bsalpha\|_\infty,\beta,0}\}^{j'}\bigg(\prod_{\substack{j\not\in {\rm supp}{(\bsnu)}\\ j>j'}} \int_{\mathbb{R}} {\rm e}^{\alpha_j|y_j|^\tau}\varphi_\beta(y_j)\,{\mathrm d}y_j\bigg).
\end{align*}
In order to ensure that the remaining factor is finite, we argue as in the proof of Proposition~\ref{prop:amin} that
\begin{align*}
\bigg(\prod_{\substack{j\not\in {\rm supp}{(\bsnu)}\\ j>j'}} \int_{\mathbb{R}} {\rm e}^{\alpha_j|y_j|^\tau}\varphi_\beta(y_j)\,{\mathrm d}y_j\bigg) \le \prod_{\substack{j\not\in {\rm supp}{(\bsnu)}\\ j>j'}} {\rm e}^{(1+\frac{\tau}{\beta}) \alpha_j} \le {\rm e}^{\sum_{j \ge 1} 2\alpha_j} = \widetilde{C} < \infty,
\end{align*}
since~$\bsalpha \in \ell^1(\bbN)$. Combining the estimates for ${\rm term_1}$ and ${\rm term_2}$ gives
\begin{align*}
    &\int_{\mathbb R^{\mathbb N}}|\bsy^{\bsnu}|\cdot|\partial^{\bsnu}F_G(\bsy_{\leq s},\mathbf 0)|\,\bsmu_{\beta}({\rm d}\bsy)\\ &\leq \widetilde{C}\|G\|_{H^{-1}(D)}\max\{1,C_{\|\bsalpha\|_\infty,\beta,0}\}^{j'}\Theta_{|\bsnu|}\bsb^{\bsnu}  \max{\{1,A_{\|\bsalpha\|_{\infty},\beta,|\bsnu|}\}}^{|\bsnu|},
\end{align*}
where $C_{\|\bsalpha\|_\infty,\beta,0}$ is defined in \eqref{eq:Cdef}.

Similarly, using~$\big(\prod_{j=1}^s {\rm e}^{\alpha_j|y_j|^\tau}\big)\big(\prod_{j>s}{\rm e}^{t\alpha_j|y_j|^\tau}\big) \le \prod_{j\geq1} {\rm e}^{\alpha_j|y_j|^\tau}$, we split the terms in \eqref{eq:term2},
\begin{align*}
    &\int_{\mathbb R^{\mathbb N}}\int_0^1 (1-t)^k|\bsy^{\bsnu}|\cdot|\partial^{\bsnu}F_G(\bsy_{\leq s},t \bsy_{>s})|\,{\rm d}t\,\bsmu_\beta({\mathrm d}\bsy)\\
    &\leq \widetilde{C}\|G\|_{H^{-1}(D)}\Theta_{|\bsnu|}\max\{1,C_{\|\bsalpha\|_\infty,\beta,0}\}^{j'}\bsb^{\bsnu} \max{\{1,A_{\|\bsalpha\|_{\infty},\beta,|\bsnu|}\}}^{|\bsnu|}.
\end{align*}

The estimates for~\eqref{eq:term1} and~\eqref{eq:term2} lead to
\begin{align}
&\bigg|\int_{\mathbb R^{\mathbb N}}(F_G(\bsy)-F_G(\bsy_{\leq s},\mathbf 0))\,\bsmu_{\beta}({\mathrm d}\bsy)\bigg|\notag\\
&\leq \widetilde{C}\|G\|_{H^{-1}(D)}\max\{1,C_{\|\bsalpha\|_\infty,\beta,0}\}^{j'} \big(\max_{2\leq\ell\leq k}(\Theta_{\ell}\max\{1,A_{\|\bsalpha\|_\infty,\beta,\ell}\}^\ell)\big)\sum_{\ell=2}^k \sum_{\substack{|\bsnu|=\ell\\ \nu_j=0~\forall j\leq s\\ \nu_j\neq 1~\forall j>s}}\bsb^{\bsnu}\notag\\%
& +\widetilde{C}\|G\|_{H^{-1}(D)}\max\{1,C_{\|\bsalpha\|_\infty,\beta,0}\}^{j'}\Theta_{k+1}(k+1)\max\{1,A_{\|\bsalpha\|_\infty,\beta,k+1}\}^{k+1} \!\!\sum_{\substack{|\bsnu|=k+1\\ \nu_j=0\,\forall j\leq s}}\bsb^{\bsnu}\notag.%
\end{align}
The desired result follows by the same steps as in the proof of~\cite[Theorem~4.1]{gk24}.\quad\endproof

\section{Application of quasi-Monte Carlo methods}\label{sec:qmcerror}
The regularity results derived in Section~\ref{sec:reg} allow the application of QMC methods to approximate integrals over solutions of PDEs equipped with random input fields, covering a wider class of parameterizations than covered in the existing literature. In particular, fat-tailed distributions corresponding to~$\beta<2$ are covered by our theory.

Let $F\!:\mathbb R^s\to \mathbb R$ be a continuous and integrable function with respect to the product measure $\bigotimes_{j=1}^s \mathcal N_{\beta}(0,1)$. We will consider approximating integral quantities
\begin{align*}
I_s(F):=\int_{\mathbb R^s}F(\bsy) \bigg(\prod_{j=1}^s \varphi_{\beta}(y_j)\bigg)\,{\rm d}\bsy=\int_{(0,1)^s}F(\Phi_{\beta}^{-1}(\boldsymbol w))\,{\rm d}\boldsymbol w,
\end{align*}
where $\varphi_{\beta}\!:\mathbb R\to\mathbb R_+$ is the probability density function defined by~\eqref{eq:betadef} and $\Phi_{\beta}^{-1}$ denotes the inverse cumulative distribution function corresponding to $\prod_{j=1}^s\varphi_{\beta}(y_j)$.

A randomly shifted rank-1 lattice rule is given by
\begin{align*}
Q_{s,n}^{\boldsymbol\Delta}(F):=\frac{1}{n} \sum_{i=1}^n F(\Phi_\beta^{-1}(\{\bst_i+\boldsymbol\Delta\})),
\end{align*}
where $\{\cdot\}$ denotes the componentwise fractional part, $\boldsymbol\Delta$ is a random shift drawn from $\mathcal U([0,1]^s)$, and the lattice points
\begin{align*}
\bst_i:=\bigg\{\frac{i\boldsymbol z}{n}\bigg\}\quad\text{for}~i\in\{1,\ldots,n\},
\end{align*}
are completely characterized by a {\em generating vector} $\boldsymbol z\in\{1,\dots,n-1\}^s$.

Let $\boldsymbol\gamma=(\gamma_{\setu})_{\setu\subseteq\{1:s\}}$ be a sequence of positive weights. We assume that the integrand $F$ belongs to a special weighted Sobolev space in $\mathbb R^s$ with bounded first order mixed partial derivatives, the norm of which is given by
\begin{align*}
\|F\|_{s,\bsgamma}^2:=\!\sum_{\setu\subseteq\{1:s\}}\!\frac{1}{\gamma_{\setu}}\!\int_{\mathbb R^{|\setu|}}\!\bigg(\int_{\mathbb R^{s-|\setu|}}\!\frac{\partial^{|\setu|}}{\partial \bsy_{\setu}}F(\bsy)\bigg(\prod_{j\in-\setu}\!\varphi_{\beta}(y_j)\bigg)\,{\rm d}\bsy_{-\setu}\bigg)^2\!\bigg(\prod_{j\in\setu}\psi^2(y_j)\bigg)\,{\rm d}\bsy_{\setu},
\end{align*}
where we define the weight functions as exponentially decaying
\begin{align*}
\psi^2(x):={\rm e}^{-\theta |x|^\tau},\quad x\in \mathbb R,
\end{align*}
with parameter $\theta>0$.

The following result, which generalizes~\cite[Theorem~5.1]{hks21}, states that generating vectors can be constructed by a \emph{component-by-component (CBC)} algorithm satisfying rigorous error bounds.

\begin{theorem}%
\label{thm:heavyqmc}
Let $0<\tau\leq\beta$ and let $F$ belong to the special weighted space over $\mathbb R^s$ with weights $\bsgamma=(\gamma_{\setu})_{\setu\subseteq\{1:s\}}$. A randomly shifted lattice rule with $n\geq 2$ points in $s$ dimensions can be constructed by a CBC algorithm such that for all $\lambda\in(1/(2r),1]$, there holds
\begin{align*}
\sqrt{\mathbb E_{\boldsymbol\Delta}|I_s(F)-Q_{s,n}^{\boldsymbol\Delta}(F)|^2}\leq \bigg(\frac{1}{\phi(n)}\sum_{\varnothing\neq\mathfrak u\subseteq\{1:s\}}\gamma_{\setu}^\lambda K^{\lambda |\setu|}(2\zeta(2r\lambda))^{|\setu|}\bigg)^{1/(2\lambda)}\|F\|_{s,\bsgamma},
\end{align*}
where $\phi$ denotes the Euler totient function and
\begin{align*}
K\!:=\!\begin{cases}
\!\eqref{eq:Kdef}~\text{with arbitrary $\theta>0$ and $r\in (\frac{1}{2},1)$}&\text{if}~0<\tau<\beta<1,\\%
\!\eqref{eq:newkdef2}~\text{with $0<\theta<\frac{1}{\beta}$ and arbitrary $r \in (\frac{1}{2},1-\frac{\theta\beta}{2})$}&\text{if}~0<\tau=\beta<1,\\
\!\eqref{eq:K-thirdcase}~\text{with arbitrary $\theta>0$ and $r\in (\frac{1}{2},1)$}&\text{if}~0<\!\tau\!<1\le\beta\\&\text{or}~1\le \tau\!<\!\beta,\\
\!\eqref{eq:newkdeffinal}~\text{with $0<\theta<\frac{1}{\beta}$ and arbitrary $r \in (\frac{1}{2},1-\frac{\theta\beta}{2})$}&\text{if}~1\leq\tau=\beta.
\end{cases}
\end{align*}
\end{theorem}
\proof We will consider the different cases separately.

\paragraph{Case $0<\tau<\beta<1$} We begin by establishing the inequality
\begin{align}\label{eq:horrible}
\exp\bigg(\frac{1}{\beta}|\Phi_\beta^{-1}(t)|^{\beta}\bigg)\leq \frac{H_{\beta,\varepsilon}}{t^{1+\varepsilon}}\quad\text{for all}~t\in(0,\tfrac12)~\text{and}~\varepsilon>0,
\end{align}
where
\begin{align*}
H_{\beta,\varepsilon}:=\frac{1}{(2\Gamma(\frac{1}{\beta}))^{1+\varepsilon}}\bigg(\frac{1-\beta}{{\rm e}\beta\varepsilon}\bigg)^{\frac{(1-\beta)(1+\varepsilon)}{\beta}}(1+\varepsilon)^{\frac{1+\varepsilon}{\beta}}.
\end{align*}
When $w<0$, there holds (see, e.g.,~\cite[Equation~5]{incomplgamma})
\begin{align*}
\Phi_{\beta}(w)=\frac{1}{2\Gamma(\frac{1}{\beta})}\int_{\frac{|w|^\beta}{\beta}}^\infty t^{\frac{1}{\beta}-1}{\rm e}^{-t}\,{\rm d}t=\frac{1}{2\Gamma(\frac{1}{\beta})}\int_{\frac{|w|^\beta}{\beta}}^\infty t^{\frac{1}{\beta}-1}{\rm e}^{-\big(1-\frac{1}{1+\varepsilon}\big)t}{\rm e}^{-\frac{1}{1+\varepsilon}t}\,{\rm d}t,
\end{align*}
where $\Gamma$ denotes the gamma function and
\begin{align*}
t^{\frac{1}{\beta}-1}{\rm e}^{-\big(1-\frac{1}{1+\varepsilon}\big)t}\leq \bigg(\frac{(1-\beta)(1+\varepsilon)}{{\rm e}\beta\varepsilon}\bigg)^{\frac{1}{\beta}-1}\quad\text{for all}~t>0~\text{and}~\varepsilon>0.
\end{align*}
Therefore
\begin{align*}
\Phi_{\beta}(w)&\leq \frac{1}{2\Gamma(\frac{1}{\beta})}\bigg(\frac{(1-\beta)(1+\varepsilon)}{{\rm e}\beta\varepsilon}\bigg)^{\frac{1}{\beta}-1}\int_{\frac{|w|^\beta}{\beta}}^\infty {\rm e}^{-\frac{1}{1+\varepsilon}t}\,{\rm d}t\\
&=\frac{1}{2\Gamma(\frac{1}{\beta})}\bigg(\frac{1-\beta}{{\rm e}\beta\varepsilon}\bigg)^{\frac{1}{\beta}-1}(1+\varepsilon)^{\frac{1}{\beta}}{\rm e}^{-\frac{1}{1+\varepsilon}\frac{|w|^\beta}{\beta}}.
\end{align*}
Since we can now write $t=\Phi_{\beta}(w)$ $\Leftrightarrow$ $w=\Phi_{\beta}^{-1}(t)$ for all $t\in(0,\frac12)$, the inequality~\eqref{eq:horrible} follows by simple algebraic manipulation.

Our next goal is to ensure that the conditions of~\cite[Lemma~3]{KSWW10} are satisfied. In particular, we let $h\geq 1$, define
\begin{align*}
\widehat\Theta(h)=\frac{2}{\pi^2h^2}\int_0^{1/2}\frac{\sin^2(\pi h t)}{\psi^2(\Phi_{\beta}^{-1}(t))\phi_{\beta}(\Phi_{\beta}^{-1}(t))}\,{\rm d}t,
\end{align*}
and proceed to derive a bound for $\widehat \Theta(h)$ in terms of $h$. To this end, we have
\begin{align}
\widehat \Theta(h)=\frac{2}{\pi^2c_{\beta}h^2}\int_0^{1/2}\sin^2(\pi h t)\exp\bigg(\theta |\Phi_{\beta}^{-1}(t)|^{\tau}+\frac{1}{\beta}|\Phi_{\beta}^{-1}(t)|^\beta\bigg)\,{\rm d}t.\label{eq:thetahatref}
\end{align}
We can bound the term~$\theta|\Phi_\beta^{-1}(t)|^\tau$ by using Young's inequality as follows:
\begin{align}
 \theta |\Phi_\beta^{-1}(t)|^\tau \le \frac{\varepsilon}{\beta} |\Phi_\beta^{-1}(t)|^\beta + \left(\frac{\varepsilon}{\tau} \right)^{1-\frac{\beta}{\beta-\tau}} \frac{\beta-\tau}{\beta} \theta^{\frac{\beta}{\beta-\tau}}.\label{eq:young}
\end{align}
This leads to the upper bound
\begin{align}
\widehat \Theta(h)&\leq \frac{2}{\pi^2c_{\beta}h^2}\exp\bigg(\bigg(\frac{\varepsilon}{\tau}\bigg)^{1-\frac{\beta}{\beta-\tau}}\frac{\beta-\tau}{\beta}\theta^{\frac{\beta}{\beta-\tau}}\bigg)\!\int_0^{1/2}\!\sin^2(\pi h t)\exp\bigg(\frac{1+\varepsilon}{\beta}|\Phi_{\beta}^{-1}(t)|^\beta\bigg)\,{\rm d}t\notag\\
&\leq \frac{2H_{\beta,\varepsilon}^{1+\varepsilon}}{\pi^2c_{\beta}h^2}\exp\bigg(\bigg(\frac{\varepsilon}{\tau}\bigg)^{1-\frac{\beta}{\beta-\tau}}\frac{\beta-\tau}{\beta}\theta^{\frac{\beta}{\beta-\tau}}\bigg)\int_0^{1/2}\sin^2(\pi h t)t^{-1-2\varepsilon-\varepsilon^2}\,{\rm d}t\notag\\
&\leq \frac{4\pi^{\varepsilon^2-2+2\varepsilon}H_{\beta,\varepsilon}^{1+\varepsilon}}{(2-\varepsilon^2-2\varepsilon)(\varepsilon^2+2\varepsilon)c_{\beta}}\exp\bigg(\bigg(\frac{\varepsilon}{\tau}\bigg)^{1-\frac{\beta}{\beta-\tau}}\frac{\beta-\tau}{\beta}\theta^{\frac{\beta}{\beta-\tau}}\bigg)h^{-2+\varepsilon^2+2\varepsilon},\label{eq:Thetabound}
\end{align}
where we applied~\eqref{eq:horrible} and~\cite[Lemma~3]{KSWW10}. By defining
\begin{align}
r=1-\frac12\varepsilon^2-\varepsilon\in(\tfrac12,1)\quad\Leftrightarrow\quad \varepsilon=\sqrt{1+2(1-r)}-1\in(0,\sqrt 2-1),\label{eq:rdef}
\end{align}
we obtain by direct application of~\cite[Theorem~8]{NicholsKuo14} for the root-mean-squared (R.M.S.) error that
\begin{align*}
{\rm R.M.S.~error}\leq\! \bigg(\frac{1}{\phi(n)}\!\sum_{\varnothing\neq\setu\subseteq\{1:s\}}\!\gamma_{\setu}^{\lambda}K^{\lambda|\setu|}(2\zeta(2r\lambda))^{|\setu|}\bigg)^{1/(2\lambda)}\|F\|_{{s,\bsgamma}}\quad\text{for all}~\!\lambda\in \big(\tfrac{1}{2r},1\big],
\end{align*}
where $r\in(\frac12,1)$ can be chosen arbitrarily and
\begin{align}
K&=\frac{4\pi^{-2r}}{(2r)(2-2r)c_{\beta}}\exp\bigg(\bigg(\frac{\sqrt{1+2(1-r)}-1}{\tau}\bigg)^{1-\frac{\beta}{\beta-\tau}}\frac{\beta-\tau}{\beta}\theta^{\frac{\beta}{\beta-\tau}}\bigg)\frac{1}{(2\Gamma(\frac{1}{\beta}))^{1+2(1-r)}}\notag\\
&\quad \times\bigg(\frac{1-\beta}{{\rm e}\beta(\sqrt{1+2(1-r)}-1)}\bigg)^{\frac{(1-\beta)(1+2(1-r))}{\beta}}(\sqrt{1+2(1-r)})^{\frac{1+2(1-r)}{\beta}}.\label{eq:Kdef}
\end{align}

\paragraph{Case $0<\tau=\beta<1$} In this case, the identity~\eqref{eq:thetahatref} can be estimated directly as
\begin{align*}
\widehat \Theta(h)&=\frac{2}{\pi^2c_{\beta}h^2}\int_0^{1/2}\sin^2(\pi h t)\exp\bigg(\frac{1+\theta\beta}{\beta}|\Phi_{\beta}^{-1}(t)|^\beta\bigg)\,{\rm d}t\\
&\leq \frac{2H_{\beta,\varepsilon}^{1+\beta\theta}}{\pi^2c_{\beta}h^2}\int_0^{1/2}\sin^2(\pi h t)t^{-(1+\theta\beta)(1+\varepsilon)}\,{\rm d}t\\
&\leq \frac{4H_{\beta,\varepsilon}^{1+\beta\theta}\pi^{\beta\theta+\varepsilon\beta\theta+\varepsilon-2}}{c_{\beta}(2-\beta\theta-\varepsilon\beta\theta-\varepsilon)(\beta\theta+\varepsilon\beta\theta+\varepsilon)}h^{-2\big(1-\frac{\beta\theta+\varepsilon\beta\theta+\varepsilon}{2}\big)}.
\end{align*}
We choose~$r = 1 - \frac{\beta \theta + \varepsilon \beta \theta + \varepsilon}{2}$. We can achieve~$r \in \left(\frac{1}{2},1\right)$ since~$\varepsilon \in (0,1)$ is arbitrarily small and~$0<\theta<\frac{1}{\beta}$. The conclusion holds with
\begin{align}\label{eq:newkdef2}K=\frac{4\pi^{2r}}{c_{\beta}(2r)(2-2r)}\frac{1}{(2\Gamma(\frac{1}{\beta}))^{3-2r}}\bigg(\frac{1-\beta}{{\rm e}\beta}\frac{1+\beta\theta}{3-2r - 1 - \beta\theta}\bigg)^{\frac{(1-\beta)(3-2r)}{\beta}}\left(\frac{3-2r}{1+\beta\theta}\right)^{\frac{3-2r}{\beta}}.
\end{align}

\paragraph{Case $0<\tau<1\le\beta$} 

For~$\beta \geq 1$, it can be shown (see \cite[Equation~5.5]{hks21}) that 
\begin{align}
\exp\left( \frac{|\Phi_{\beta}^{-1}(t)|^{\beta}}{\beta} \right) \le \frac{1}{t} \quad \quad\text{for all}~t \in (0,\tfrac{1}{2}).\label{eq:ge-one-bound}
\end{align}
As in the previous case, since~$\beta > \tau$, Young's inequality~\eqref{eq:young} %
leads to the upper bound~\eqref{eq:Thetabound} with~$H_{\beta,\varepsilon} = 1$.
Thus, choosing~$r$ as in~\eqref{eq:rdef}, the conclusion holds with 
\begin{align}
K&=\frac{4\pi^{-2r}}{(2r)(2-2r)c_{\beta}}\exp\bigg(\bigg(\frac{\sqrt{1+2(1-r)}-1}{\tau}\bigg)^{1-\frac{\beta}{\beta-\tau}}\frac{\beta-\tau}{\beta}\theta^{\frac{\beta}{\beta-\tau}}\bigg).\label{eq:K-thirdcase}
\end{align}

\paragraph{Case $1 \le \tau=\beta$} In this case, the identity~\eqref{eq:thetahatref} can be estimated directly as
\begin{align*}
\widehat \Theta(h)&=\frac{2}{\pi^2c_{\beta}h^2}\int_0^{1/2}\sin^2(\pi h t)\exp\bigg(\frac{1+\theta\beta}{\beta}|\Phi_{\beta}^{-1}(t)|^\beta\bigg)\,{\rm d}t\\
&\leq \frac{4\pi^{\beta\theta-2}}{c_{\beta}(2-\beta\theta)\beta\theta}h^{-2\big(1-\frac{\beta\theta}{2}\big)},
\end{align*}
using~\eqref{eq:ge-one-bound}  and~\cite[Lemma~3]{KSWW10}, which additionally imposes the constraint $0<\theta<\frac{1}{\beta}$. Thus, choosing~$r = 1 - \frac{\beta \theta }{2}$, the conclusion holds with
\begin{align}\label{eq:newkdeffinal}K=\frac{4\pi^{-2r}}{c_{\beta}(2r)(2-2r)}.
\end{align}

\paragraph{Case $1\le \tau<\beta$} Choosing~$r$ as in~\eqref{eq:rdef}, using Young's inequality it can be shown that the conclusion holds with~$K$ as in~\eqref{eq:K-thirdcase}.\quad\endproof

We obtain the following as a corollary (compare with~\cite[Theorem~6.1]{hks21}).

\begin{theorem}\label{thm:weights}
Let $u_s$ denote the dimensionally-truncated PDE solution $u$, which satisfies assumptions~{\rm \ref{assump2}},~{\rm \ref{assump3}}, and~{\rm \ref{assump5}}. If $\beta\neq \tau$, we choose $\theta\in(2\|\bsalpha\|_{\infty},\infty)$; if $\beta=\tau$, we additionally require $\|\bsalpha\|_{\infty}<\frac{1}{2\beta}$ and take $\theta\in(2\|\bsalpha\|_{\infty},\frac{1}{\beta})$. Define the product and order dependent (POD) weights by setting
\begin{align*}
\gamma_{\setu}=\bigg((|\setu|!)^\sigma \prod_{j\in\setu}\frac{ (C+1)b_j}{(\theta-2\alpha_j)^{\frac{1}{2\tau}}\sqrt{K^\lambda \zeta(2r\lambda)\,\Gamma(1+\frac{1}{\tau})}}\bigg)^{\frac{2}{1+\lambda}},\quad \setu\subseteq\{1:s\},
\end{align*}
where $K$ is specified in Theorem~\ref{thm:heavyqmc} and, if $\beta\neq \tau$,
\begin{align*}
\lambda=\begin{cases}
\frac{p}{2-p}&\text{if}~\frac23<p<\frac{1}{\sigma},\\
\frac{1}{2-2\delta}~\text{for arbitrary}~\delta\in(0,\frac12)&\text{if}~0<p\leq \min\{\frac23,\frac{1}{\sigma}\},~p\neq \frac{1}{\sigma}.
\end{cases}
\end{align*}
Otherwise, if $\beta=\tau,$ \begin{align*}
\lambda=\begin{cases}
\frac{p}{2-p}\quad\text{provided that additionally}~\theta<\frac{3p-2}{p\beta}&\text{if}~\frac23<p<\frac{1}{\sigma},\\
\frac{1}{2-\theta\beta-2\delta}~\text{for arbitrary}~\delta\in(0,\frac12-\frac12\theta\beta)&\text{if}~0<p\leq \min\{\frac23,\frac{1}{\sigma}\},~p\neq \frac{1}{\sigma}.
\end{cases}
\end{align*}
 Then a randomly shifted lattice rule with $n\geq 2$ points in $s$ dimensions can be constructed by a CBC algorithm for the integrand $F=G\circ u_s$  such that 
\begin{align*}
\sqrt{\mathbb E_{\boldsymbol\Delta}|\mathbb E(F)-Q_{s,n}^{\boldsymbol\Delta}(F)|^2}\leq C' \begin{cases}
    \phi(n)^{-\min\{1/p-1/2,1-\delta\}}, &\text{if}~\tau <\beta,\\
    \phi(n)^{-\min\{1/p-1/2,1-\frac{\theta\beta}{2}-\delta\}}, &\text{if}~\tau =\beta,
\end{cases}
\end{align*}
where $C'>0$ is independent of the dimension $s$.
\end{theorem}
\proof By Theorem~\ref{thm:weak-rand}, there holds $|\partial^{\bsnu}F(\bsy)|\lesssim (C+1)^{|\bsnu|}(|\bsnu|!)^\sigma \bsb^{\bsnu}\prod_{j\geq 1}{\rm e}^{\alpha_j|y_j|^\tau }$. We begin by estimating the weighted Sobolev norm:
\begin{align*}
&\|F\|_{s,\bsgamma}^2= \sum_{\setu\subseteq\{1:s\}}\frac{1}{\gamma_{\setu}}\int_{\mathbb R^{|\setu|}}\bigg(\int_{\mathbb R^{s-|\setu|}}\partial_{\setu}F(\bsy)\bigg(\prod_{j\in-\setu}\varphi_{\beta}(y_j)\bigg)\,{\rm d}\bsy_{-\setu}\bigg)^2\bigg(\prod_{j\in\setu}\psi_j^2(y_j)\bigg)\,{\rm d}\bsy_{\setu}\\
&\lesssim \!\!\sum_{\setu\subseteq\{1:s\}}\!\!\frac{(|\setu|!)^{2\sigma}(C+1)^{2|\setu|}\bsb_{\setu}^2}{\gamma_{\setu}}\bigg(\prod_{j\in-\setu}\int_{\mathbb R}{\rm e}^{2\alpha_j|y_j|^\tau}\varphi_\beta(y_j)\,{\rm d}y_j\bigg)\prod_{j\in\setu}\int_{\mathbb R}{\rm e}^{2\alpha_j |y_j|^\tau-\theta |y_j|^\tau}\,{\rm d}y_j.
\end{align*}
Here, we can use the identity
\begin{align*}
\int_{\mathbb R}{\rm e}^{2\alpha_j|y_j|^\tau-\theta|y_j|^\tau}\,{\rm d}y_j=\frac{2}{(\theta-2\alpha_j)^\frac{1}{\tau}\Gamma(1+\frac{1}{\tau})},\quad \theta>2\alpha_j,
\end{align*}
and for $\alpha_j<\frac{1}{4\tau}$ the upper bound
\begin{align*}
\int_{\mathbb R}{\rm e}^{2\alpha_j|y_j|^\tau}\varphi_{\beta}(y_j)\,{\rm d}y_j\leq \frac{{\rm e}^{\frac{\beta-\tau}{\beta}2\alpha_j}}{(1-2\tau\alpha_j)^{1/\beta}}\leq \exp\bigg(2\,\frac{\beta-\tau}{\beta}\alpha_j\bigg)\exp\bigg(\frac{1}{\beta}\frac{2\tau\alpha_j}{1-2\tau\alpha_j}\bigg);
\end{align*}
see proof of Proposition~\ref{prop:amin} in the special case $q=2$. 
Let $j'\in\mathbb R$ be an index such that $\alpha_j<\frac{1}{4\tau}$ for all $j>j'$. Moreover, define $C':=\prod_{j=1}^{j'}\int_{\mathbb R}{\rm e}^{2\alpha_j|y_j|}\varphi_\beta(y_j)\,{\rm d}y_j<\infty$. Then
\begin{align*}
\|F\|_{s,\bsgamma}^2 &\lesssim (C')^2\sum_{\setu\subseteq\{1:s\}}\frac{(|\setu|!)^{2\sigma}\bsb_{\setu}^2}{\gamma_{\setu}}\prod_{j\in\setu}\frac{2(C+1)^2}{(\theta-2\alpha_j)^{\frac{1}{\tau}}\Gamma(1+\frac{1}{\tau})}\\
&\quad\quad\quad\quad \times \prod_{j\in\{1:s\}\setminus(\setu\cup \{1:j'\})}\exp\bigg(2\,\frac{\beta-\tau}{\beta}\alpha_j\bigg)\exp\bigg(\frac{1}{\beta}\frac{2\tau\alpha_j}{1-2\tau\alpha_j}\bigg)\\
&\leq (C')^2\exp\bigg(2\,\frac{\beta-\tau}{\beta}\sum_{j> j'}\alpha_j + \frac{1}{\beta}\sum_{j> j'}\frac{2\tau\alpha_j}{1-2\tau \alpha_j}\bigg)\\&\quad\quad\quad\quad\times\sum_{\setu\subseteq\{1:s\}}\frac{(|\setu|!)^{2\sigma}\boldsymbol b_{\setu}^2}{\gamma_{\setu}}\prod_{j\in\setu}\frac{2(C+1)^2}{(\theta-2\alpha_j)^{\frac{1}{\tau}}\Gamma(1+\frac{1}{\tau})}.
\end{align*}

The QMC R.M.S.~error can therefore be bounded by
\begin{align*}
{\rm R.M.S.~error}\lesssim \bigg(\frac{1}{\phi(n)}\bigg)^{\frac{1}{2\lambda}}C(\lambda,\boldsymbol\gamma,s),
\end{align*}
where the constant
\begin{align*}
C(\lambda,\boldsymbol\gamma,s)&:=\bigg(\sum_{\varnothing\neq \setu\subseteq\{1:s\}}\gamma_{\setu}^\lambda K^{\lambda |\setu|}(2\zeta(2r\lambda))^{|\setu|}\bigg)^{\frac{1}{2\lambda}}\\
&\quad\quad\times \bigg(\sum_{\setu\subseteq\{1:s\}}\frac{(|\setu|!)^{2\sigma}}{\gamma_{\setu}}\prod_{j\in\setu}\frac{2(C+1)^2b_j^2}{(\theta-2\alpha_j)^{\frac{1}{\tau}}\Gamma(1+\frac{1}{\tau})}\bigg)^{\frac12}
\end{align*}
can be minimized in complete analogy to~\cite{kss12} by choosing the weights $\gamma_{\varnothing}=1$ and
\begin{align*}
\gamma_{\setu}=\bigg((|\setu|!)^\sigma \prod_{j\in\setu}\frac{ (C+1)b_j}{(\theta-2\alpha_j)^{\frac{1}{2\tau}}\sqrt{K^\lambda \zeta(2r\lambda)\,\Gamma(1+\frac{1}{\tau})}}\bigg)^{\frac{2}{1+\lambda}},\quad \varnothing\neq\setu\subset \mathbb N.
\end{align*}
On the other hand, plugging these weights into the expression for $C(\lambda,\boldsymbol\gamma,s)$ yields
\begin{align*}
&[C(\lambda,\boldsymbol\gamma,s)]^{\frac{2\lambda}{1+\lambda}}\\
&=\sum_{\setu\subseteq\{1:s\}}2^{|\setu|}\zeta(2r\lambda)^{\frac{1}{1+\lambda}|\setu|}(|\setu|!)^{\frac{2\sigma\lambda}{1+\lambda}}K^{\frac{\lambda}{\lambda+1}|\setu|}\frac{[(C+1)^2]^{\frac{\lambda}{1+\lambda}|\setu|}}{(\theta-2\alpha_j)^{\frac{\lambda}{\tau(1+\lambda)}|\setu|}\Gamma(1+\frac{1}{\tau})^{\frac{\lambda}{1+\lambda}|\setu|}}\prod_{j\in\setu}b_j^{\frac{2\lambda}{1+\lambda}}\\
&=\sum_{\ell=0}^s2^\ell \zeta(2r\lambda)^{\frac{1}{1+\lambda}\ell}(\ell!)^{\frac{2\sigma\lambda}{1+\lambda}}K^{\frac{\lambda}{\lambda+1}\ell}\frac{[(C+1)^2]^{\frac{\lambda}{1+\lambda}\ell}}{(\theta-2\alpha_j)^{\frac{\lambda}{\tau(1+\lambda)}\ell}\Gamma(1+\frac{1}{\tau})^{\frac{\lambda}{1+\lambda}\ell}}\sum_{\substack{|\setu|=\ell\\ \setu\subseteq\{1:s\}}}\prod_{j\in\setu}b_j^{\frac{2\lambda}{1+\lambda}}\\
&\leq \sum_{\ell=0}^\infty 2^\ell \zeta(2r\lambda)^{\frac{1}{1+\lambda}\ell}(\ell!)^{\frac{2\sigma\lambda}{1+\lambda}}K^{\frac{\lambda}{\lambda+1}\ell}\frac{[(C+1)^2]^{\frac{\lambda}{1+\lambda}\ell}}{(\theta-2\alpha_j)^{\frac{\lambda}{\tau(1+\lambda)}\ell}\Gamma(1+\frac{1}{\tau})^{\frac{\lambda}{1+\lambda}\ell}}\sum_{\substack{|\setu|=\ell\\ \setu\subseteq\mathbb N}}\prod_{j\in\setu}b_j^{\frac{2\lambda}{1+\lambda}}\\
&\leq \sum_{\ell=0}^\infty 2^\ell \zeta(2r\lambda)^{\frac{1}{1+\lambda}\ell}(\ell!)^{\frac{2\sigma\lambda}{1+\lambda}-1}K^{\frac{\lambda}{\lambda+1}\ell}\frac{[(C+1)^2]^{\frac{\lambda}{1+\lambda}\ell}}{(\theta-2\alpha_j)^{\frac{\lambda}{\tau(1+\lambda)}\ell}\Gamma(1+\frac{1}{\tau})^{\frac{\lambda}{1+\lambda}\ell}}\bigg(\sum_{j\geq 1}b_j^{\frac{2\lambda}{1+\lambda}}\bigg)^\ell,
\end{align*}
where we used the inequality $\sum_{|\setu|=\ell,~\setu\subseteq\mathbb N}\prod_{j\in\setu}c_j\leq \frac{1}{\ell!}\big(\sum_{j\geq 1}c_j\big)^\ell$ for $c_j>0$ for all $j\geq 1$.

There are now two cases to consider:
\begin{enumerate}
\item If $\frac{2}{3}<p<\frac{1}{\sigma}$ and $\beta\neq\tau$, we take $\lambda=\frac{p}{2-p}$. Then $\lambda\in(\frac{1}{2},1)$ and the ratio test shows that $C(\lambda,\boldsymbol\gamma,s)$ can be bounded independently of $s$. If $\beta=\tau$, then we additionally need $\lambda\in(\frac{1}{2-\beta\theta-\varepsilon\beta\theta-\varepsilon},1)$ which can be accomplished by choosing $\theta<\frac{3p-2-\varepsilon p}{\beta p(1+\varepsilon)}$.
\item First we assume that $\beta>\tau$. If $p\in(0,\min\{\frac23,\frac{1}{\sigma}\})$, we take $\lambda=\frac{1}{2-2\delta}$ for arbitrary $\delta\in(0,\frac12)$. Then $\frac{2\lambda}{1+\lambda}=\frac{2}{3-2\delta}\in(\frac23,1)$ and we can use the inequality $\|\boldsymbol b\|_{\ell^{2\lambda/(1+\lambda)}}\leq \|\boldsymbol b\|_{\ell^p}$ to obtain
\begin{align*}
&[C(\lambda,\boldsymbol\gamma,s)]^{\frac{2\lambda}{1+\lambda}}\\
&\leq \sum_{\ell=0}^\infty 2^\ell \zeta(2r\lambda)^{\frac{1}{1+\lambda}\ell}(\ell!)^{\frac{2\sigma\lambda}{1+\lambda}-1}\frac{K^{\frac{\lambda}{\lambda+1}\ell}[(C+1)^2]^{\frac{\lambda}{1+\lambda}\ell}}{(\theta-2\alpha_j)^{\frac{\lambda}{\tau(1+\lambda)}\ell}\Gamma(1+\frac{1}{\tau})^{\frac{\lambda}{1+\lambda}\ell}}\bigg(\sum_{j\geq 1}b_j^{p}\bigg)^{\frac{2\lambda\ell}{(1+\lambda)p}}.
\end{align*}
An application of the ratio test shows that $C(\lambda,\boldsymbol\gamma,s)$ can be bounded independently of $s$ in this case as well. In the case~$\beta = \tau$ we take~$\lambda = \frac{1}{2 - \theta\beta - 2\delta}$ for arbitrary~$\delta \in (0,\frac{1}{2}-\frac{\theta\beta}{2})$. Then~$\frac{2\lambda}{1+\lambda} = \frac{2}{3-\theta\beta - 2\delta} \in (\frac{2}{3-\theta\beta},1)$ and we can proceed as above.
\end{enumerate}
This concludes the proof.\quad\endproof
\section{Total error}\label{sec:totalerror}

The overall error can be decomposed as~\eqref{eq:overallerror}, i.e., it consists of the \emph{dimension truncation} error, the \emph{spatial discretization error}, and the \emph{cubature error}. Choosing a finite element method for the spatial discretization and a randomly shifted rank-1 lattice rule as quadrature method, we shall derive a convergence result for the combined overall error~\eqref{eq:overallerror}. For this purpose, we denote by~$V_h$ be the space of continuous, piecewise linear functions on~$D$ satisfying the boundary condition. The space~$V_h$ is of dimension~$M_h <\infty$, where~$M_h$ is of order~$h^d$,~$d\in\{1,2,3\}$.

\begin{theorem}\label{thm:all}
Let the assumptions~{\rm \ref{assump2}}--{\rm \ref{assump8}} hold for some~$0<t\le 1$. Then, it is possible to generate a QMC rule using a CBC algorithm under the assumptions of~Theorem~\ref{thm:weights} satisfying the combined error estimate
\begin{align*}
&\sqrt{\mathbb E_{\boldsymbol\Delta}\|\mathbb E[u]-Q_{s,n}^{\boldsymbol\Delta}(u)\|_{H_0^1(D)}^2}
 \le C\! \left( s^{-\frac{2}{p}+1} \!+\!  h^{s'} \!\!+\! \phi(n)^{-\min\{1/p-1/2,1-\delta\}} \right),\!\hspace*{.55cm}\text{ if }\tau < \beta,\\%
 &\sqrt{\mathbb E_{\boldsymbol\Delta}\|\mathbb E[u]-Q_{s,n}^{\boldsymbol\Delta}(u)\|_{H_0^1(D)}^2}
 \le C\! \left( s^{-\frac{2}{p}+1} \!+ \! h^{s'} \!\!+\! \phi(n)^{-\min\{1/p-1/2,1-\frac{\theta\beta}{2} - \delta\}} \right), \!\text{ if }\tau = \beta
\end{align*}
for all~$0< s' < t$, where~$\phi(n)$ is the Euler totient function and the constant $C>0$ is independent of $s$, $h$, and $n$. If assumption~{\rm \ref{assump7}} holds with~$t=1$ the result holds with~$s' = 1$.
\end{theorem}
\proof
    The result follows from the error decomposition~\eqref{eq:overallerror} as follows: Theorem~\ref{thm:dimtrunc} is applied to the dimension truncation error. For the spatial discretization error~\cite[Theorems~2.1 and~2.2]{Teckentrup2013} are applied. Finally,~\cite[Theorem~6.5]{GKKSS24} is used to extend Theorems~\ref{thm:heavyqmc} and~\ref{thm:weights} to~$H_0^1(D)$-valued integrands.\quad\endproof

\section{Numerical experiments}\label{sec:numex}
Our theoretical framework covers not only lognormally distributed random inputs, but also generalized $\beta$-Gaussian distributions and Gevrey regular nonlinearities. The case where $a$ is lognormally parameterized has been considered in many papers such as \cite{log, log2, log3, log4, Kuo2016ApplicationOQ, KuoNuyens2018}, which include extensive numerical examples. The more general setup is demonstrated in the following numerical experiments. 
To this end, we fix the spatial domain $D=(0,1)^2$ for the PDE~\eqref{eq:pdemodel} equipped with an input random field given by~\eqref{eq:newadef}, 
where we set $\psi_j(\bsx)=0.5j^{-\vartheta}\sin(\pi j x_1)\sin(\pi j x_2)$, $\vartheta>1$, $f(\bsx)=x_2$, and $h(y)=\frac{y}{1+\sqrt{|y|}}$. Moreover, we choose
\begin{align*}
\xi(y)=\begin{cases}\exp(-y^{-2})&\text{if}~y\neq 0,\\ 0&\text{if}~y=0.\end{cases}
\end{align*}
It is a consequence of~\cite{mathstack} and~\cite[pg.~16]{qi1996general}, together with Stirling's approximation $\nu^\nu< \frac{\nu!\exp\big(\nu-\frac{1}{12\nu}+\frac{1}{360\nu^3}\big)}{\sqrt{2\pi \nu}}$, that
\begin{align*}
\bigg|\frac{{\rm d}^\nu}{{\rm d}z^{\nu}}\xi(z)\bigg|\leq 3^\nu(\nu!)^{3/2}\quad\text{for all}~z\in \mathbb R~\text{and}~\nu\in\mathbb N_0.
\end{align*}

In addition, we set~$\beta=\frac12$ for the distribution of the input $\beta$-Gaussian distribution and $\alpha_j=\|\psi_j\|_{L^\infty(D)}$. We assess the QMC error, dimension truncation error, and finite element discretization error numerically for this problem using two different values $\vartheta\in\{1.75,2.00\}$ for the decay rate of the parametric input. In all experiments, we use the CBC construction as detailed in~\cite{NicholsKuo14} with the input weights derived in Theorem~\ref{thm:weights} to obtain the $s$-dimensional QMC rules
\begin{align*}
Q^{(r)} u =  \frac{1}{n}\sum_{i=1}^n u(\cdot,\Phi_{\beta}^{-1}(\{\bst_i+\boldsymbol\Delta^{(r)}\})),\quad r\in\{1,\ldots,R\},
\end{align*}
where $\boldsymbol \Delta^{(1)},\ldots,\boldsymbol\Delta^{(R)}\overset{\rm i.i.d.}{\sim}\mathcal U([0,1]^s)$. The generating vector is obtained in all experiments using the parameters $\tau=\beta$, $\theta=1.001$, $r=0.70$, $\delta=0.05$. As our estimator of the expected value $\mathbb E[u]$, we take
\begin{align*}
\overline{Q}u=\frac{1}{R}\sum_{r=1}^R Q^{(r)}u.
\end{align*}

\begin{figure}[!t]
\centering
\begin{tikzpicture}
    \begin{loglogaxis}[
        width=.46\textwidth,
        height=.46\textwidth,
        xlabel={number of nodes $n$},
        ylabel={R.M.S.~error},
        legend pos=north east,
        grid=both,
        grid style={dash pattern=on 1pt off 1pt on 1pt off 1pt},
        xmin=10, xmax=100000,
        ymin=5e-8, ymax=1e-3,
        xtick={10,100,1000,10000,100000},
        ytick={1e-8,1e-7,1e-6,1e-5,1e-4,1e-3},
        legend entries={$H_0^1(D)$ error, $L^2(D)$ error}
    ]
    \addplot[
        color=blue,
        mark=* ,
        only marks
    ] coordinates {
        (17, 0.000326005)
        (31, 0.000266371)
        (67, 0.000161875)
        (127, 0.0000664711)
        (263, 0.0000346036)
        (503, 0.0000226511)
        (1013, 0.0000122649)
        (2003, 0.0000107224)
        (4003, 4.36918e-6)
        (8009, 3.79533e-6)
        (16007, 2.15499e-6)
        (32009, 1.37072e-6)
        (63997, 1.29591e-6)
    };
    
    \addplot[
        color=purple,
        mark=square*,
        only marks
    ] coordinates {
        (17, 0.0000514072)
        (31, 0.0000410483)
        (67, 0.0000182265)
        (127, 9.50792e-6)
        (263, 4.5225e-6)
        (503, 2.47737e-6)
        (1013, 1.21854e-6)
        (2003, 1.08122e-6)
        (4003, 4.42421e-7)
        (8009, 3.03237e-7)
        (16007, 1.53036e-7)
        (32009, 1.14162e-7)
        (63997, 9.48077e-8)
    };
    
    \addplot[
        color=blue,
        domain=17:63997,
        samples=100,
        dashed,
        line width=1pt
    ] { 0.00240776*x^(-0.720539)};
    
    \addplot[
        color=purple,
        domain=17:63997,
        samples=100,
        dashed,
        line width=1pt
    ] {0.000489024*x^(-0.815358)};
    
    \end{loglogaxis}
\end{tikzpicture}\begin{tikzpicture}
    \begin{loglogaxis}[
        width=.46\textwidth,
        height=.46\textwidth,
        xlabel={number of nodes $n$},
        ylabel={R.M.S.~error},
        legend pos=north east,
        grid=both,
        grid style={dash pattern=on 1pt off 1pt on 1pt off 1pt},
        xmin=10, xmax=100000,
        ymin=2e-8, ymax=1e-3,
        xtick={10,100,1000,10000,100000},
        ytick={1e-8,1e-7,1e-6,1e-5,1e-4,1e-3},
        legend entries={$H_0^1(D)$ error, $L^2(D)$ error}
    ]
    \addplot[
        color=blue,
        mark=* ,
        only marks
    ] coordinates {
        (17, 0.000318374)
        (31, 0.000222384)
        (67, 0.000095622)
        (127, 0.0000958866)
        (263, 0.0000427763)
        (503, 0.0000186633)
        (1013, 0.0000123024)
        (2003, 8.48447e-6)
        (4003, 4.95602e-6)
        (8009, 2.11487e-6)
        (16007, 1.71415e-6)
        (32009, 8.93827e-7)
        (63997, 5.71268e-7)
    };
    
    \addplot[
        color=purple,
        mark=square*,
        only marks
    ] coordinates {
        (17, 0.0000562378)
        (31, 0.0000352126)
        (67, 0.0000137663)
        (127, 0.0000119276)
        (263, 6.65799e-6)
        (503, 2.79086e-6)
        (1013, 1.49104e-6)
        (2003, 8.79013e-7)
        (4003, 4.69262e-7)
        (8009, 2.39471e-7)
        (16007, 1.90631e-7)
        (32009, 8.5124e-8)
        (63997, 4.14295e-8)
    };
    
    \addplot[
        color=blue,
        domain=17:63997,
        samples=100,
        dashed,
        line width=1pt
    ] { 0.00307305*x^(-0.783656)};
    
    \addplot[
        color=purple,
        domain=17:63997,
        samples=100,
        dashed,
        line width=1pt
    ] {0.000667005*x^(-0.8679)};
    
    \end{loglogaxis}
    \end{tikzpicture}\caption{Left: the obtained R.M.S.~errors for the experiment with $\vartheta=1.75$. The obtained least squares fits are $\mathcal O(n^{-0.72})$ and $\mathcal O(n^{-0.82})$ for the $H_0^1(D)$ and $L^2(D)$ errors, respectively. Right: the obtained R.M.S.~errors for the experiment with $\vartheta=2.00$. The obtained empirical convergence rates are $\mathcal O(n^{-0.78})$ and $\mathcal O(n^{-0.87})$ for the $H_0^1(D)$ and $L^2(D)$ errors, respectively.}\label{fig:1}
\end{figure}

We begin by assessing the QMC cubature error for the PDE problem. We fix the dimension $s=100$ and solve the PDE~\eqref{eq:pdemodel} using a first-order finite element method with mesh width $h=2^{-7}$. We consider the R.M.S.~errors
\begin{align*}
\sqrt{\frac{1}{R(R-1)}\sum_{r=1}^R \|\overline{Q}u-Q^{(r)}u\|_{H_0^1(D)}^2}\quad\text{and}\quad \sqrt{\frac{1}{R(R-1)}\sum_{r=1}^R \|\overline{Q}u-Q^{(r)}u\|_{L^2(D)}^2}
\end{align*}
with $R=16$ random shifts. The results are displayed in Figure~\ref{fig:1}. In all cases, we observe faster-than-Monte Carlo cubature convergence rates consistent with our theory. Specifically, Theorem~\ref{thm:heavyqmc} ensures that the theoretically expected QMC convergence rates are $\mathcal O(n^{-0.70})$ for both $\vartheta=1.75$ and $\vartheta=2.00$, while we observe the rates $\mathcal O(n^{-0.72})$ and $\mathcal O(n^{-0.78})$ for the $H_0^1(D)$ errors, respectively. We note that the observed rates $\mathcal O(n^{-0.82})$ and $\mathcal O(n^{-0.87})$ for the $L^2(D)$ errors corresponding to $\vartheta=1.75$ and $\vartheta=2.00$ are slightly better.

In order to analyze the dimension truncation error, we fix $h=2^{-7}$ as the mesh size and take the PDE solution corresponding to truncation dimension $s'=256$ as the reference solution. We consider the quantities
\begin{align*}
&\bigg\|\int_{\mathbb R^{s'}}u_{s',h}(\cdot,\bsy)\,\bsmu_{\beta}({\rm d}\bsy)-\int_{\mathbb R^{s}}u_{s,h}(\cdot,\bsy)\,\bsmu_{\beta}({\rm d}\bsy)\bigg\|_{H_0^1(D)}\\
&\text{and}\quad \bigg\|\int_{\mathbb R^{s'}}u_{s',h}(\cdot,\bsy)\,\bsmu_{\beta}({\rm d}\bsy)-\int_{\mathbb R^{s}}u_{s,h}(\cdot,\bsy)\,\bsmu_{\beta}({\rm d}\bsy)\bigg\|_{L^2(D)}
\end{align*}
for $s\in\{2,4,8,16,32,64\}$, where the high-dimensional integrals appearing inside the norms are approximated using lattice rules with $n=63\,997$ nodes subject to a single random shift. The results are displayed in Figure~\ref{fig:dimtrunc}.  In this case, the expected dimension truncation rates are $\mathcal O(s^{-2.5})$ and $\mathcal O(s^{-3.00})$ for $\vartheta=1.75$ and $\vartheta=2.00$, respectively. We observe the empirical rates $\mathcal O(s^{-2.34})$ and $\mathcal O(s^{-2.64})$ for the $H_0^1(D)$ dimension truncation errors corresponding to $\vartheta=1.75$ and $\vartheta=2.00$, while the respective empirical $L^2(D)$ dimension truncation rates are $\mathcal O(s^{-2.38})$ and $\mathcal O(s^{-2.84})$. The numerical results agree well with our theory, with the minor discrepancies potentially explained by the numerical integration and finite element errors present in the computations.

\begin{figure}[!t]
\centering
\begin{tikzpicture}
    \begin{loglogaxis}[
        width=.46\textwidth,
        height=.46\textwidth,
        xlabel={dimension $s$},
        ylabel={truncation error},
        legend pos=north east,
        grid=both,
        grid style={dash pattern=on 1pt off 1pt on 1pt off 1pt},
        xmin=1, xmax=100,
        ymin=1e-8, ymax=1e-3,
        xtick={1,10,100},
        ytick={1e-8,1e-7,1e-6,1e-5,1e-4,1e-3,1e-2},
        xticklabels={1,10,100},
        legend entries={$H_0^1(D)$ error, $L^2(D)$ error}
    ]
    \addplot[
        color=blue,
        mark=* ,
        only marks
    ] coordinates {
        (2, 0.000375987)
        (4, 0.0000749514)
        (8, 0.0000131155)
        (16, 2.58924e-6)
        (32, 4.98524e-7)
        (64, 1.25051e-7)
    };
    
    \addplot[
        color=purple,
        mark=square*,
        only marks
    ] coordinates {
        (2, 0.0000708481)
        (4, 0.0000146209)
        (8, 2.70178e-6)
        (16, 5.07331e-7)
        (32, 9.79461e-8)
        (64, 1.93011e-8)
    };
    
    \addplot[
         color=blue,
         domain=2:64,
         samples=100,
         dashed,
         line width=1pt
     ] {0.00181376*x^( -2.33734)};
    
     \addplot[
         color=purple,
         domain=2:64,
         samples=100,
         dashed,
         line width=1pt
     ] {0.000379031*x^(-2.37964)};
    
    \end{loglogaxis}
    \end{tikzpicture}\begin{tikzpicture}
    \begin{loglogaxis}[
        width=.46\textwidth,
        height=.46\textwidth,
        xlabel={dimension $s$},
        ylabel={truncation error},
        legend pos=north east,
        grid=both,
        grid style={dash pattern=on 1pt off 1pt on 1pt off 1pt},
        xmin=1, xmax=100,
        ymin=1e-9, ymax=1e-3,
        xtick={1,10,100},
        ytick={1e-9,1e-8,1e-7,1e-6,1e-5,1e-4,1e-3,1e-2},
        xticklabels={1,10,100},
        legend entries={$H_0^1(D)$ error, $L^2(D)$ error}
    ]
    \addplot[
        color=blue,
        mark=* ,
        only marks
    ] coordinates {
        (2, 0.000185212)
        (4, 0.0000263139)
        (8, 3.4757e-6)
        (16, 5.91883e-7)
        (32, 8.10375e-8)
        (64, 2.27752e-8)
    };
    
    \addplot[
        color=purple,
        mark=square*,
        only marks
    ] coordinates {
        (2, 0.0000344079)
        (4, 5.12595e-6)
        (8, 7.06334e-7)
        (16, 9.49538e-8)
        (32, 1.31309e-8)
        (64, 1.93304e-9)
    };
    
    \addplot[
         color=blue,
         domain=2:64,
         samples=100,
         dashed,
         line width=1pt
     ] {0.000992308*x^(-2.64372)};
    
     \addplot[
         color=purple,
         domain=2:64,
         samples=100,
         dashed,
         line width=1pt
     ] {0.000252765*x^(-2.83769)};
    
    \end{loglogaxis}
    \end{tikzpicture}\caption{Left: the obtained dimension truncation errors for the experiment with $\vartheta=1.75$. The obtained least squares fits are $\mathcal O(s^{-2.34})$ and $\mathcal O(s^{-2.38})$ for the $H_0^1(D)$ and $L^2(D)$ errors, respectively. Right: the obtained dimension truncation errors for the experiment with $\vartheta=2.00$. The obtained least squares fits are $\mathcal O(s^{-2.64})$ and $\mathcal O(s^{-2.84})$ for the $H_0^1(D)$ and $L^2(D)$ errors, respectively.}\label{fig:dimtrunc}
\begin{tikzpicture}
    \begin{loglogaxis}[
        width=.46\textwidth,
        height=.46\textwidth,
        xlabel={mesh width $h$},
        ylabel={discretization error},
        legend pos=south east,
        grid=both,
        grid style={dash pattern=on 1pt off 1pt on 1pt off 1pt},
        xmin=0.01, xmax=1,
        ymin=1e-5, ymax=0.1,
        xtick={1,0.1,0.01},
         xticklabels={$1^{\phantom{0}}$, $10^{-1}$, $10^{-2}$},
        ytick={1e-6,1e-5,1e-4,1e-3,1e-2,1e-1},
        minor xtick = {2e-2,3e-2,4e-2,5e-2,6e-2,7e-2,8e-2,9e-2,2e-1,3e-1,4e-1,5e-1,6e-1,7e-1,8e-1,9e-1},
        legend entries={$H_0^1(D)$ error, $L^2(D)$ error}
    ]
    \addplot[
        color=blue,
        mark=* ,
        only marks
    ] coordinates {
        (0.5, 0.0750393)
        (0.25, 0.0443396)
        (0.125, 0.0232939)
        (0.0625, 0.0117527)
        (0.03125, 0.00576524)
        (0.015625, 0.00258246)
    };
    
    \addplot[
        color=purple,
        mark=square*,
        only marks
    ] coordinates {
        (0.5, 0.0120797)
        (0.25, 0.00405564)
        (0.125, 0.00108698)
        (0.0625, 0.000275867)
        (0.03125, 0.0000669664)
        (0.015625, 0.0000137972)
    };
    
     \addplot[
         color=blue,
         domain=0.015625:0.5,
         samples=100,
         dashed,
         line width=1pt
     ] {0.16438*x^(0.974872)};
    
     \addplot[
         color=purple,
         domain=0.015625:0.5,
         samples=100,
         dashed,
         line width=1pt
     ] {0.056771*x^(1.96027)};
    
    \end{loglogaxis}
    \end{tikzpicture}\begin{tikzpicture}
    \begin{loglogaxis}[
        width=.46\textwidth,
        height=.46\textwidth,
        xlabel={mesh width $h$},
        ylabel={discretization error},
        legend pos=south east,
        grid=both,
        grid style={dash pattern=on 1pt off 1pt on 1pt off 1pt},
        xmin=0.01, xmax=1,
        ymin=1e-5, ymax=0.1,
        xtick={1,0.1,0.01},
        xticklabels={$1^{\phantom{0}}$, $10^{-1}$, $10^{-2}$},
        ytick={1e-5,1e-4,1e-3,1e-2,1e-1},
        minor xtick = {2e-2,3e-2,4e-2,5e-2,6e-2,7e-2,8e-2,9e-2,2e-1,3e-1,4e-1,5e-1,6e-1,7e-1,8e-1,9e-1},
        legend entries={$H_0^1(D)$ error, $L^2(D)$ error}
    ]
    \addplot[
        color=blue,
        mark=* ,
        only marks
    ] coordinates {
        (0.5, 0.0752405)
        (0.25, 0.0443627)
        (0.125, 0.0232952)
        (0.0625, 0.0117589)
        (0.03125, 0.00576801)
        (0.015625, 0.00258365)
    };
    
    \addplot[
        color=purple,
        mark=square*,
        only marks
    ] coordinates {
        (0.5, 0.0121017)
        (0.25, 0.00403319)
        (0.125, 0.0010822)
        (0.0625, 0.000274751)
        (0.03125, 0.000066531)
        (0.015625, 0.0000137023)
    };
    
     \addplot[
         color=blue,
         domain=0.015625:0.5,
         samples=100,
         dashed,
         line width=1pt
     ] {0.164686*x^(0.975315)};
    
     \addplot[
         color=purple,
         domain=0.015625:0.5,
         samples=100,
         dashed,
        line width=1pt
     ] {0.0567907*x^(1.96217)};
    
    \end{loglogaxis}
    \end{tikzpicture}\caption{Left: the obtained expected finite element discretization errors for the experiment with $\vartheta=1.75$. The obtained least squares fits are $\mathcal O(h^{0.97})$ and $\mathcal O(h^{1.96})$ for the $H_0^1(D)$ and $L^2(D)$ errors, respectively. Right: the obtained expected finite element discretization errors for the experiment with $\vartheta=2.00$. The obtained least squares fits are $\mathcal O(h^{0.98})$ and $\mathcal O(h^{1.96})$ for the $H_0^1(D)$ and $L^2(D)$ errors, respectively.}\label{fig:femfig}
\end{figure}

Finally, we assess the expected finite element error by fixing the truncation dimension $s=100$ and using the finite element solution corresponding to $h'=2^{-7}$ as the reference solution. We approximate the quantities
\begin{align*}
&\bigg\|\int_{\mathbb R^s}(u_{s,h'}(\cdot,\bsy)-u_{s,h}(\cdot,\bsy))\,\bsmu_{\beta}({\rm d}\bsy)\bigg\|_{H_0^1(D)}\\
&\text{and}\quad \bigg\|\int_{\mathbb R^s}(u_{s,h'}(\cdot,\bsy)-u_{s,h}(\cdot,\bsy))\,\bsmu_{\beta}({\rm d}\bsy)\bigg\|_{L^2(D)}
\end{align*}
by using a hierarchy of regular finite element meshes with $h=2^{-k}$, $k\in\{1,2,3,4,5,6\}$. The high-dimensional integrals appearing inside the norms were approximated by a QMC cubature rule using $n=63\,997$ nodes subject to a single random shift. The results are displayed in Figure~\ref{fig:femfig}. The observed finite element errors align nearly perfectly with the theoretical $H_0^1(D)$ error rate $\mathcal O(h)$ and $L^2(D)$ error rate $\mathcal O(h^2)$.

\section*{Conclusions}
In this paper, we have considered the application of QMC to the analysis of PDEs characterized by Gevrey regular, generalized Gaussian input uncertainty. The parametric regularity estimates yield dimension-independent QMC convergence rates for elliptic PDEs belonging to this class, extending the applicability of QMC methods to high-dimensional settings involving, e.g., fat-tailed input distributions. Moreover, our numerical experiments validate the sharpness of the derived error bounds.

\section*{Acknowledgements} The authors thank the anonymous referees for their comments, which helped improve this paper.
\section*{Appendix}
\renewcommand{\thesection}{A}
We require the following technical result for the numerical experiments.
\begin{lemma} Let $\bsnu\in\mathscr F$, $\bsnu\leq \mathbf 1$, and suppose that $\psi_j\in L^\infty(D)$ for $j\geq 1$ such that $(\|\psi_j\|_{L^\infty(D)})_{j\geq 1}\in\ell^1(\mathbb N)$. Furthermore, let $h\!:\mathbb R\to \mathbb R$ be a continuously differentiable function such that $\sup_{x\in\mathbb R}|h'(x)|\leq 1$ and let $\xi\!:\mathbb R\to\mathbb R$ be an infinitely many times continuously differentiable function satisfying the bound
\begin{align*}
\bigg|\frac{{\rm d}^k}{{\rm d}z^k}\xi(z)\bigg|\leq C_{\xi}^k (k!)^{\sigma}\quad\text{for all}~z\in \mathbb R~\text{and}~k\in\mathbb N_0
\end{align*}
for some $C_{\xi},\sigma\geq 1$. Define
\begin{align}\label{eq:newadef}
&a(\bsx,\bsy)=\exp\bigg(\sum_{j\geq 1}h(y_j)\psi_j(\bsx)\bigg)\exp\bigg(\xi\bigg(\sum_{j\geq 1}y_j\psi_j(\bsx)\bigg)\bigg),\quad \bsx\in D,~\bsy\in U_{\bsalpha,\tau}.
\end{align}
Then
\begin{align*}
\bigg\|\frac{\partial^{\bsnu}a(\cdot,\bsy)}{a(\cdot,\bsy)}\bigg\|_{L^\infty(D)}\leq \bigg(\frac{{\rm e}}{2}\bigg)^{\sigma}(|\bsnu|!)^{\sigma}\boldsymbol b^{\bsnu},
\end{align*}
where $\boldsymbol b=(b_j)_{j\geq 1}$ is defined by $b_j=2^{\sigma}C_{\xi}\|\psi_j\|_{L^\infty(D)}$ for $j\geq 1$.
\end{lemma}
\proof We begin by noting that the infinite series in~\eqref{eq:newadef} are summable since our assumption that $\sup_{x\in\mathbb R}|h'(x)|\leq 1$ implies $|h(x)|\leq |x|$ for all $x\in\mathbb R$ by the mean value theorem. The claim is trivially true if $\bsnu=\mathbf 0$, so we let $\bsnu\in\mathscr F\setminus\{\mathbf 0\}$ with $\bsnu\leq\mathbf 1$. By the Leibniz product rule, we obtain
\begin{align}
\partial^{\bsnu}a(\bsx,\bsy)=\sum_{\boldsymbol m\leq \bsnu}\binom{\bsnu}{\boldsymbol m}\partial^{\boldsymbol m}\exp\bigg(\xi\bigg(\sum_{j\geq 1}y_j\psi_j(\bsx)\bigg)\bigg)\partial^{\bsnu-\boldsymbol m}\exp\bigg(\sum_{j\geq 1}h(y_j)\psi_j(\bsx)\bigg).\label{sf0}
\end{align}
It is straightforward to estimate
\begin{align}
\bigg|\partial^{\bsnu-\boldsymbol m}\exp\bigg(\sum_{j\geq 1}h(y_j)\psi_j(\bsx)\bigg)\bigg|\leq \exp\bigg(\sum_{j\geq 1}h(y_j)\psi_j(\bsx)\bigg)\boldsymbol\rho^{\bsnu-\boldsymbol m},\label{sf1}
\end{align}
where $\boldsymbol\rho=(\rho_j)_{j\geq1}$ with $\rho_j=\|\psi_j\|_{L^\infty(D)}$. On the other hand, by Fa\`a di Bruno's formula (cf., e.g.,~\cite{savits}), there holds
\begin{align}\label{eq:fdib}
\partial^{\boldsymbol m}\exp\bigg(\xi\bigg(\sum_{j\geq 1}y_j\psi_j(\bsx)\bigg)\bigg)=\exp\bigg(\xi\bigg(\sum_{j\geq 1}y_j\psi_j(\bsx)\bigg)\bigg)\sum_{\lambda=1}^{|\boldsymbol m|}\tau_{\boldsymbol m,\lambda}(\bsx,\bsy),
\end{align}
where the sequence $(\tau_{\boldsymbol m,\lambda}(\bsx,\bsy))$ is defined recursively by
\begin{align*}
&\tau_{\boldsymbol m,0}\equiv \delta_{\boldsymbol m,\mathbf 0},\\
&\tau_{\boldsymbol m,\lambda}\equiv 0\quad\text{if}~|\boldsymbol m|<\lambda~\text{or}~\lambda<0,~\text{and}\\
&\tau_{\boldsymbol m+\boldsymbol e_j,\lambda}(\bsx,\bsy)=\sum_{\boldsymbol w\leq \boldsymbol m}\binom{\boldsymbol m}{\boldsymbol w}\partial^{\boldsymbol w+\boldsymbol e_j}\bigg(\xi\bigg(\sum_{j\geq 1}y_j\psi_j(\bsx)\bigg)\bigg)\tau_{\boldsymbol m-\boldsymbol w,\lambda-1}(\bsx,\bsy)
\end{align*}
otherwise. It is not difficult to see that
\begin{align*}
\partial^{\boldsymbol m}\xi\bigg(\sum_{j\geq 1}y_j\psi_j(\bsx)\bigg)=\frac{{\rm d}^{|\boldsymbol m|}}{{\rm d}z^{|\boldsymbol m|}}\xi(z)\bigg|_{z=\sum_{j\geq 1}y_j\psi_j(\bsx)}\prod_{j\in{\rm supp}(\boldsymbol m)}\psi_j(\bsx)^{m_j},
\end{align*}
which implies that
\begin{align*}
\sup_{\bsy\in U_{\bsalpha,\tau}}\!\|\tau_{\boldsymbol m+\boldsymbol e_j,\lambda}(\cdot,\bsy)\|_{L^\infty}\!\leq\! \!\sum_{\boldsymbol w\leq \boldsymbol m}\!\binom{\boldsymbol m}{\boldsymbol w}C_{\xi}^{|\boldsymbol w|+1}((|\boldsymbol w|\!+\!1)!)^{\sigma}\boldsymbol \rho^{\boldsymbol w}\!\!\sup_{\bsy\in U_{\bsalpha,\tau}}\!\!\!\|\tau_{\boldsymbol m-\boldsymbol w,\lambda-1}(\cdot,\bsy)\|_{L^\infty}.
\end{align*}
It is now a consequence of~\cite[Lemma~5.1]{ks24} that
\begin{align*}
\sup_{\bsy\in U_{\bsalpha,\tau}}\|\tau_{\boldsymbol m,\lambda}(\cdot,\bsy)\|_{L^\infty}\leq C_{\xi}^{|\boldsymbol m|}\bigg(\frac{|\boldsymbol m|!(|\boldsymbol m|-1)!}{\lambda!(|\boldsymbol m|-\lambda)!(\lambda-1)!}\bigg)^{\sigma}\boldsymbol\rho^{\boldsymbol m}.
\end{align*}
This implies that~\eqref{eq:fdib} can be bounded by
\begin{align}
&\bigg|\partial^{\boldsymbol m}\exp\bigg(\xi\bigg(\sum_{j\geq 1}y_j\psi_j(\bsx)\bigg)\bigg)\bigg|\notag\\
&\leq \exp\bigg(\xi\bigg(\sum_{j\geq 1}y_j\psi_j(\bsx)\bigg)\bigg) C_{\xi}^{|\boldsymbol m|}\boldsymbol\rho^{\boldsymbol m}\sum_{\lambda=1}^{|\boldsymbol m|}\bigg(\frac{|\boldsymbol m|!(|\boldsymbol m|-1)!}{\lambda!(|\boldsymbol m|-\lambda)!(\lambda-1)!}\bigg)^{\sigma}\notag\\
&\leq \exp\bigg(\xi\bigg(\sum_{j\geq 1}y_j\psi_j(\bsx)\bigg)\bigg) C_{\xi}^{|\boldsymbol m|}\boldsymbol\rho^{\boldsymbol m}(|\boldsymbol m|!(|\boldsymbol m|-1)!)^{\sigma}\bigg(\sum_{\lambda=1}^{|\boldsymbol m|}\frac{1}{(|\boldsymbol m|-\lambda)!(\lambda-1)!}\bigg)^{\sigma}\notag\\
&= \exp\bigg(\xi\bigg(\sum_{j\geq 1}y_j\psi_j(\bsx)\bigg)\bigg) 2^{\sigma |\boldsymbol m|-\sigma}C_{\xi}^{|\boldsymbol m|}\boldsymbol\rho^{\boldsymbol m}(|\boldsymbol m|!)^{\sigma},\label{sf2}
\end{align}
where we used the inequality $\sum_k a_k\leq \big(\sum_k a_k^{1/\beta}\big)^\beta$ for $a_k\geq 0$ and $\beta\geq 1$ and the identity $\sum_{\lambda=1}^m \frac{1}{(m-\lambda)!(\lambda-1)!}=\frac{2^{m-1}}{(m-1)!}$ (see~\cite[Lemma~A.1]{ks24}). Plugging~\eqref{sf1} and~\eqref{sf2} into~\eqref{sf0} yields
\begin{align*}
\bigg\|\frac{\partial^{\bsnu}a(\cdot,\bsy)}{a(\cdot,\bsy)}\bigg\|_{L^\infty(D)}\leq 2^{\sigma |\bsnu|-\sigma}C_{\xi}^{|\bsnu|}\boldsymbol\rho^{\bsnu}\sum_{\boldsymbol m\leq\bsnu}\binom{\bsnu}{\boldsymbol m}(|\boldsymbol m|!)^{\sigma}.
\end{align*}
Finally, we can estimate using the Vandermonde convolution that
\begin{align*}
&\sum_{\boldsymbol m\leq\bsnu}\binom{\bsnu}{\boldsymbol m}(|\boldsymbol m|!)^{\sigma}=\sum_{\ell=0}^{|\bsnu|}(\ell!)^{\sigma}\sum_{\substack{|\boldsymbol m|=\ell\\ \boldsymbol m\leq \bsnu}}\binom{\bsnu}{\boldsymbol m}\leq \sum_{\ell=0}^{|\bsnu|}(\ell!)^{\sigma}\sum_{\substack{|\boldsymbol m|=\ell\\ \boldsymbol m\leq \bsnu}}\binom{\bsnu}{\boldsymbol m}^{\sigma}\\
&\leq \bigg(\sum_{\ell=0}^{|\bsnu|}\ell!\sum_{\substack{|\boldsymbol m|=\ell\\ \boldsymbol m\leq \bsnu}}\binom{\bsnu}{\boldsymbol m}\bigg)^{\sigma}
\leq (|\bsnu|!)^{\sigma}\bigg(\sum_{\ell=0}^{|\bsnu|}\frac{1}{(|\bsnu|-\ell)!}\bigg)^{\sigma}\leq {\rm e}^{\sigma}(|\bsnu|!)^{\sigma},
\end{align*}
as desired.\quad\endproof

{\em Remark.} In the special case when $h(x)=x$, the above result holds for all $\bsnu\in\mathscr F$.

\bibliographystyle{plain}
\bibliography{gevrey_arXiv}

\end{document}